\documentclass[final]{siamltex}
\usepackage{a4}
\usepackage{amsmath}
\usepackage{amssymb}
\usepackage{graphicx}
\usepackage{enumerate}

\def\+{\oplus}

\newcommand{\dt}{\Delta t}

\newcommand{\R}{{\mathbb R}}

\newcommand{\ch}{{\mathcal H}}

\newcommand{\cI}{{\mathcal I}}
\newcommand{\cJ}{{\mathcal J}}

\newcommand{\cC}{{\mathcal C}}
\newcommand{\cH}{{\mathcal H}}

\newcommand{\cT}{{\mathcal T}}

\newcommand{\cK}{{\mathcal K}}

\newcommand{\diver}{{\rm{div}}}
\newcommand{\T}{{\mathbb{T}}}

\def\vfi{\varphi}
\def\de{\delta}
\def\la{\lambda}
\def\rife#1{(\ref{#1})}

\newcommand{\ds}{\displaystyle}

\newtheorem{remark}{\textbf{Remark}}

\title{Convergence of a finite difference scheme to weak solutions of the system of partial differential equation arising
 in  mean field games}
 \author{Yves Achdou \thanks { Univ. Paris Diderot, Sorbonne Paris Cit{\'e}, Laboratoire Jacques-Louis Lions, UMR 7598, UPMC, CNRS, F-75205 Paris, France. 
  achdou@ljll-univ-paris-diderot.fr} \and
 Alessio Porretta  \thanks{ Universit{\`a} di Roma Tor Vergata,
Dipartimento di Matematica,
Via della Ricerca Scientifica 1,
00133 Roma, Italy. porretta@axp.mat.uniroma2.it
}
 }
\begin{document}

\maketitle
\begin{abstract}
Mean field type models describing the limiting behavior of stochastic differential games as the 
number of players tends to $+\infty$, have been recently introduced by J-M. Lasry and P-L. Lions.
Under suitable assumptions, they lead to a system of two coupled partial differential equations,
 a forward Bellman equation and a backward Fokker-Planck equations.
Finite difference schemes  for the approximation of such systems have been proposed in previous works. 
Here, we prove the convergence of these schemes towards a  weak solution of the system  of partial differential equations.  
\end{abstract}



\section {Introduction}
Mean field type models describing the asymptotic behavior of  stochastic differential games (Nash equilibria) as the number of players
tends to $+\infty$ have recently been introduced by J-M. Lasry and P-L. Lions
 \cite{MR2269875,MR2271747,MR2295621}, and termed {\sl mean field games} by the same authors.
 Related ideas have been developed independently in the engineering
literature  by Huang-Caines-Malham{\'e}, see for example \cite{MR2352434}. For brevity, the acronym {\sl MFG} will sometimes be used for {\sl mean field games}.
Examples of MFG models with applications in economics and social sciences are proposed in \cite{MR2762362,ABLLM14}.
\\
\\  The simplest MFG model lead
 to systems of  evolutive partial differential equations involving  two unknown scalar functions: the density of the agents in a given state $x \in \R^d$, 
namely  $m=m(t,x)$ and the value function  $u=u(t,x)$.
 The present work is devoted to  finite difference schemes for the systems of partial differential equations. 
 Although the methods and the theoretical results obtained below can be easily  generalized,
 the present work focuses on the two-dimensional case for the following reasons:
1) the one dimensional case is easier and allows too special arguments; 2) in dimension two,  the description of the discrete methods discussed below remain fairly simple.
 Besides, several important applications of the mean field games theory are two-dimensional, in particular those related to crowd dynamics.
\\ \\
In the state-periodic setting, typical MFG  model comprises  the following system of  partial differential equations in $(0,T)\times  \T^2$
  \begin{eqnarray}
  \label{eq:1}
&\frac {\partial  u}{\partial t} (t,x) -\nu \Delta  u  (t,x) + H(x,\nabla  u (t,x))  = 
 F(m(t,x)), \\
\label{eq:2}
& \frac {\partial  m}{\partial t}  (t,x)+\nu \Delta  m (t,x) +\diver\left(  m (t,x) 
\frac {\partial H} {\partial p} (x,\nabla  u(t,x) )\right)   = 0 ,
\end{eqnarray}with the initial and terminal conditions
\begin{equation}
  \label{eq:3} u(0,x)=u_0(x),\quad  m(T,x) = m_T(x),\quad \hbox{in } \T^2,
\end{equation}
given a  cost function $u_0$ and a probability density  $m_T$.
\\
Here, we denote by $\T^2=[0,1]^2$ the $2-$dimensional unit torus,
 and by $\Delta$, $\nabla$ and $\diver$, respectively, the Laplace, the gradient
and the divergence operator acting on the state variable $x$.
The parameter $\nu$ is the diffusion coefficient.
Hereafter, we will always assume that $\nu>0$. 
The system also involves the scalar Hamiltonian $H(x,p)$, 
which is assumed to be continuous, convex  and  $\cC^1$ regular with respect to $p$.
 The notation $\frac{\partial H}{\partial p}(x,q)$ is used for the gradient of $p\mapsto H(x,p)$ at $p=q$. Finally, in  the term  $F(m(t,x))$,  $F$ is a continuous real valued function defined on $\R_+$.    Hereafter the notation  $Q$ will be used for the space-time cylinder $(0,T)\times \T^2$.
\\
\\
We have chosen to focus on the case when the cost $u_0$  depends directly on $x$. In some realistic situations, 
the final cost may depend on the density of the players, i.e. $u_{|t=0}= \Phi_0[m_{|t=0}](x)$, where 
 $\Phi_0$ is an operator acting on probability densities, which may be local or not.  We will not tackle this aspect, in order to keep the discussion as simple as possible. 
Similarly, by working on the torus $\T^2$, we avoid the discussion of the boundary conditions,  but other boundary value problems with for example Dirichlet or Neumann conditions could be considered.  It is also possible to consider different initial conditions than in (\ref{eq:3}): if there is a condition of the type $m(t=0,\cdot)=m_0$ instead of $u(t=0,\cdot)=u_0$, then the system models a planning problem, see \cite{PLL} for a description of the model and mathematical results, and  \cite{MR3090129,MR3195848} for new existence and uniqueness results. 
\\
System (\ref{eq:1})-(\ref{eq:2}) consists then of a forward Bellman equation coupled with a backward Fokker-Planck equation. The forward-backward structure is an important feature of  this system, which makes it necessary to design new strategies for its mathematical analysis (see \cite{MR2271747,MR2295621}) and  for numerical approximation.
The main results on the mathematical analysis of  (\ref{eq:1})-(\ref{eq:2}) are contained in the pioneering articles \cite{MR2271747,MR2295621}, but 
many important aspects  of the theory developed by  J-M. Lasry and P-L. Lions on MFG are not published in journals or books. They  can nevertheless be found in the videos of  the lectures of P-L. Lions (in French) at Coll{\`e}ge de France:  see \cite{PLL}.  A very good introduction is also given in the notes by P. Cardaliaguet, \cite{cardaliaguet2010}, with a special emphasis on the deterministic case, i.e. $\nu=0$ in (\ref{eq:1})-(\ref{eq:2}). The survey of Gomes et al \cite{MR3195844} also addresses  interesting extensions of the model, 
and the so-called {\sl master equation} first introduced in \cite{PLL}. 

 Depending on the data and on $F$ and $H$, different notions of solutions can be relevant for (\ref{eq:1})-(\ref{eq:3}): indeed, if the right hand side of (\ref{eq:1}) is replaced by $\Phi[m(t,\cdot)](x)$ where $\Phi$ is a nonlocal smoothing operator,  mapping probability measures on $\T^2$ to  $\cC ^1$ functions, if $H$ depends smoothly on $x$ and if the data $u_0$ and $m_T$ are smooth, then classical solutions can be found, see  \cite{MR2271747,MR2295621}.
The same is true if e.g. $H$ is Lipschitz continuous w.r.t. its second argument $p$ and $F$ in (\ref{eq:1}) is a continuous function.  The situation is different in the case when $H$ has a strictly superlinear growth with respect to $p$ and $F$ is a continuous function: in this case, one has to look for weak solutions, see \cite{MR2295621} and the recent article \cite{MR3305653} which is devoted to  weak solutions to Fokker-Planck equations and to the system (\ref{eq:1})-(\ref{eq:2}).
\\
 Since the (semi-)analytic solutions of the MFG system do not exist in general,   any attempt to apply MFG models and to  get  qualitative/quantitative information from them
 must rely on numerical simulations and scientific computing. 
Therefore,  the research has also been active on numerical methods for approximating (\ref{eq:1})-(\ref{eq:3}):
 a numerical method based  on the reformulation of the model as an optimal control problem 
for the Fokker-Planck equation  with an application in economics was proposed in  \cite{MR2647032}. Discrete time, finite state space mean  field games
were discussed in  \cite{MR2601334}. We also refer to \cite{MR2974160,MR2928382}
 for a specific constructive approach when the Hamiltonian is quadratic.
 Finally,   semi-Lagrangian approximations have been studied in \cite{MR3148086,CS2014}.\\
 The finite difference method described and studied  below has first been proposed and discussed in \cite{MR2679575,MR2888257}.
It will be reviewed in \S~\ref{sec:finite-diff-schem}. The numerical scheme  basically relies on monotone approximations of the Hamiltonian and on a suitable weak formulation of the Fokker-Planck equation. It has several important features:  
\\
$\phantom{aaa}\bullet$ existence and uniqueness for the discretized problems can be obtained by similar arguments as those used in the continuous case,\\
$\phantom{aaa}\bullet$ they are robust when $\nu\to 0$ (the deterministic limit of the models),\\
$\phantom{aaa}\bullet$ bounds on the solutions, which are uniform in the grid step, can be proved under reasonable assumptions on the data.\\
A first result on  the convergence to classical solutions was given in \cite{MR2679575}. The issue of convergence  was studied with more details in \cite{MR3097034,MR3135339}:
in these works, the starting point/assumption was the existence of a classical solution of (\ref{eq:1})-(\ref{eq:3}).  The proof of convergence mainly consisted in plugging the classical solution into the system of equations arising from the finite difference method, and use the consistency and stability properties of the scheme in order to get estimates and pass to the limit.
\\
In the present work, the goal is different: we wish to prove that as the grid steps tend to zero, the solution of the discretized MFG system converges to a weak solution of (\ref{eq:1})-(\ref{eq:3}), without assuming the existence of the latter;  so this work  will supply as a by-product a new strategy for proving the existence of weak solutions. One key step will be to obtain a priori estimates on the solutions to the discrete systems, and these will mainly come from the fact that the structure of the MFG system is preserved by the chosen finite difference method. This step will be achieved in \S~\ref{sec:priori-estimates}.
Note that \S~\ref{sec:priori-estim-discr} is concerned with a priori estimates for the discrete version of the Fokker-Planck equation (\ref{eq:2}). These estimates may have their own interest, independently from MFG models, and  may be put in relation with recent works of Gallou{\"e}t et al \cite{MR2904585} in the context of finite volume methods. Once these estimates are obtained, the most important difficulty will be to pass to the limit in the discrete Bellman equation. The strategy to that purpose is to first prove some $L^1 $ compactness of the sequence of state-gradients of the discrete solutions,   then to adapt some techniques proposed by Boccardo, Murat and Puel, see \cite{MR766873}, for studying weak solutions of (\ref{eq:1}). This will done in \S~\ref{l1-comp} and \ref{sec:dis-to-con}.  Here also, we think that the passage to the limit in the discrete Bellman equation may have an interest for  itself.

 \section{Finite difference schemes}
\label{sec:finite-diff-schem}
In the present paragraph, we discuss the finite difference method originally proposed in \cite{MR2679575}. 
\\
 Let $N_T$ be a positive integer and $\dt=T/{N_T}$, $t_n= n \dt$, $n=0,\dots, N_T$.
 Let $\T^2_h$ be a uniform grid on the torus with mesh step $h$, (assuming that $1/{h}$ is an integer $N_h$),
 and $x_{ij}$ denote a generic point in  $\T^2_h$.
The values of $u$ and $m$  at $(x_{i,j},t_n)$ are respectively approximated by $u^n_{i,j}$ and $m^n_{i,j}$. 
Let $u^n$ (resp. $m^n$) be the vector containing the values $u^n_{i,j}$  (resp. $m^n_{i,j}$), for  $0\le i,j< N_h$ indexed in the lexicographic order.
Hereafter, such vectors will be termed {\sl grid functions  on $\T_{h}^2$} or simply grid functions.
For all grid functions $z$, all  $i$ and $j$, we agree that $z_{i,j}= z_{(i \hbox{ mod } N_h), (j \hbox{ mod }N_h)}$.
\paragraph{Elementary finite difference operators}
Let us introduce the elementary  finite difference operators 
  \begin{equation}
\label{eq:8}
 (D_1^+ u )_{i,j} = \frac{ u_{i+1,j}-u_{i,j}   } {h} \quad \hbox{and }\quad  (D_2^+ u )_{i,j} = \frac{ u_{i,j+1}-u_{i,j}   } {h},
\end{equation}
and define $ D_h u$ as the grid function with values in $\R^2$: 
\begin{equation}
  \label{eq:36}
 (D_h u)_{i,j} =\Bigl( (D_1^+ u )_{i,j}, (D_2^+ u )_{i,j} \Bigr) \in \R^2.
\end{equation}
Let
$[\nabla_h u]_{i,j}$ be the collection of the four possible one sided finite differences at $x_{i,j}$:
\begin{equation}
  \label{eq:9}
 [\nabla_h u]_{i,j} =\Bigl((D_1^+ u )_{i,j} , (D_1^+ u )_{i-1,j}, (D_2^+ u )_{i,j}, (D_2^+ u )_{i,j-1}\Bigr) \in \R^4.
\end{equation}
We will also need the standard five point discrete Laplace operator 
\begin{displaymath}
(\Delta_h u)_{i,j}=  -\frac 1 {h^2} (4u_{i,j} -u_{i+1,j}-u_{i-1,j}-u_{i,j+1}-u_{i,j-1}).  
\end{displaymath}
For a set $v= (v^n)_{n=0,\dots,N_T}$,
where $v^n$ is grid functions  on $\T_{h}^2$,
 it will be convenient to define the family of grid functions:
 \begin{equation}
   \label{eq:35}
\partial_{t,\dt} v\equiv  \left(\frac {v^{n+1}-v^n} {\dt}\right)_{n=0,\dots, N_T-1}.
 \end{equation}
\paragraph{Numerical Hamiltonian}
In order to approximate the term $H(x, \nabla u)$ in (\ref{eq:1}), we consider a  numerical Hamiltonian $g: \T^2 \times \R^4\to \R$,  $(x,q_1,q_2,q_3,q_4)\mapsto g\left(x,q_1,q_2,q_3,q_4\right)$.
Hereafter we will often assume that the following conditions hold:
\begin{description}
\item ($\mathbf{g_1}$)  \emph{monotonicity}: $g$ is nonincreasing with respect to $q_1$ and $q_3$  and nondecreasing with respect to $q_2$ and $q_4$.
\item ($\mathbf{g_2}$)  \emph{consistency:}
  $g\left(x,q_1,q_1,q_2,q_2\right)=H(x,q), \quad \forall x\in \T^2, \forall q=(q_1,q_2)\in \R^2. $
\item ($\mathbf{g_3}$) \emph{regularity}: $g$ is  continuous and of class $\cC^1$ w.r.t.   $(q_1,q_2,q_3,q_4)$.
\item ($\mathbf{g_4}$)  \emph{convexity} : $(q_1,q_2,q_3,q_4)\mapsto g\left(x,q_1,q_2,q_3,q_4\right)$ is convex.
\end{description}
We will approximate $H(\cdot, \nabla u) (x_{i,j})$ by $g(x_{i,j}, [\nabla_h u]_{i,j} )$.\\
 Standard examples of numerical Hamiltonians fulfilling these requirements are provided by  Lax-Friedrichs or upwind schemes, see \cite{MR2679575}. 
For Hamiltonians of the form  $H(x,p)= \ch(x) + |p|^\beta$, $\beta\in (1,\infty)$,  we may choose 
\begin{equation}\label{eq:10}
  g(x,q)= \ch(x)+ G(q_1^-,q_2^+, q_3^-, q_4^+),
\end{equation}
where, for a real number $r$, $r^+=\max(r,0)$ and $r^-=\max(-r,0)$ and where
$G:(\R_+)^4\to \R_+$ is given by
\begin{equation}\label{eq:11}
  G(p)= |p|^{\beta}=(p_1^2+p_2^2+p_3^2+p_4^2)^{\frac \beta 2}.
\end{equation}
\paragraph{Discrete Bellman equation}
The discrete version of the Bellman equation is obtained by applying a semi-implicit Euler scheme to (\ref{eq:1}),  
\begin{equation}\label{eq:18}
\ds   \frac {u^{n+1}_{i,j}- u^{n}_{i,j}} {\dt}  -\nu (\Delta_h u^{n+1})_{i,j} +
 g(x_{i,j}, \left[\nabla_h u^{n+1}\right]_{i,j} )=   F(m^{n}_{i,j}),
\end{equation}
for all points in $\T_h^2$ and all $n$, $0\le n < N_T$, where all the discrete operators have been introduced above. 
Given $(m^n)_{n=0,\dots, N_T-1}$, 
(\ref{eq:18}) and the initial condition $u_{i,j}^0= u_0(x_{i,j})$ for all $(i,j)$ completely characterizes  $(u^n)_{0\le n\le N_T}$.
\paragraph{Discrete Fokker-Planck equation}
In order to approximate equation \eqref{eq:2}, it is convenient to  consider its weak formulation which involves in particular the term
\[\ds \int_{\T^2} \diver\left(m \frac {\partial H}{\partial p}(x,\nabla u) \right)  w(x)\, dx.\]
By periodicity, 
\[\ds  \int_{\T^2} \diver\left(m \frac {\partial H}{\partial p}(x,\nabla u) \right)   w(x)\, dx= - \int_{\T^2} m(x) \frac {\partial H}{\partial p}(x,\nabla u(x)) \cdot \nabla w(x)\, dx\]
holds for any test function $w$. The right hand side in the identity above  will be approximated by
\[ - h^2\sum_{i,j}    m_{i,j} \nabla_q g (x_{i,j},[\nabla_h u]_{i,j}) \cdot [\nabla_h w]_{i,j}
= h^2 \sum_{i,j} \cT_{i,j}(u,m) w_{i,j},\]
where the  transport  operator $\cT$ is defined as follows:
\begin{equation}
\label{eq:19}
  \begin{split}
  &  \cT_{i,j}(u,m)= \\ & \frac 1 h\left(
  \begin{array}[c]{l}
\ds
\left(
  \begin{array}[c]{l}
   \ds  m_{i,j}  \frac {\partial g} {\partial q_1} (x_{i,j}, [\nabla_h u]_{i,j})
     - m_{i-1,j}  \frac {\partial g} {\partial q_1}(x_{i-1,j}, [\nabla_h u]_{i-1,j}) \\ \ds
     +m_{i+1,j}  \frac {\partial g} {\partial q_2} (x_{i+1,j},[\nabla_h u]_{i+1,j})
-   \ds  m_{i,j}  \frac {\partial g} {\partial q_2} (x_{i,j}, [\nabla_h u]_{i,j})
  \end{array}
\right)
\\ +\\
\ds 
\left(
  \begin{array}[c]{l}
\ds  m_{i,j}  \frac {\partial g} {\partial q_3} (x_{i,j}, [\nabla_h u]_{i,j})
-  \ds  m_{i,j-1}\frac {\partial g} {\partial q_3} (x_{i,j-1}, [\nabla_h u]_{i,j-1}) \\
+ \ds  m_{i,j+1}\frac {\partial g} {\partial q_4} (x_{i,j+1}, [\nabla_h u]_{i,j+1})
-  m_{i,j}  \frac {\partial g} {\partial q_4} (x_{i,j}, [\nabla_h u]_{i,j})
  \end{array}
\right)
  \end{array}
\right).
  \end{split}
\end{equation}
The discrete version of equation (\ref{eq:2}) is chosen as follows:
\begin{equation}
\label{eq:20}
\ds   \frac {m^{n+1}_{i,j}- m^{n}_{i,j}} {\dt} +\nu (\Delta_h m^{n})_{i,j}
+\cT_{i,j}(u^{n+1},m^n)= 0,
\end{equation}
for all $n=0,\dots, N_T-1$.
This scheme is implicit w.r.t. to $m$ and explicit w.r.t. $u$ because the considered Fokker-Planck equation is backward.
Given $u$ this is a system of linear equations for $m$.
We introduce the compact and convex set
\begin{equation}
\label{eq:12}
    \cK_h=\{ (m_{i,j})_{ 0\le i,j <N_h}:  h^2 \sum_{i,j} m_{i,j}=1;\quad   m_{i,j}\ge 0 \}
\end{equation}
which can be viewed as the set of the discrete probability measures.
It is easy to see that if $m^n$ satisfies (\ref{eq:20}) for $0\le n<N_T$ and if $m^{N_T}\in \cK_h$, then $m^n\in \cK_h$ for all $n$,  $0\le n<N_T$.
\begin{remark}\label{sec:numer-schem-refeq:7-2}
An important property of $\cT$ is that the operator $m\mapsto  \bigl(-\nu (\Delta_h m)_{i,j} - \cT_{i,j}(u,m)\bigr)_{i,j}$
 is the adjoint of the linearized version of the operator
 $u\mapsto   \bigl( -\nu (\Delta_h u)_{i,j} + g(x_{i,j}, [\nabla_h u]_{i,j} )\bigr)_{i,j}$.
\\ This property implies that the structure of (\ref{eq:1})-(\ref{eq:2}) is preserved in the discrete version (\ref{eq:18})-(\ref{eq:20}). In particular,
 it implies the uniqueness result stated in Theorem~\ref{sec:summary-1} below.
\end{remark}
\paragraph{Summary}
The fully discrete scheme
 for  system (\ref{eq:1}),(\ref{eq:2}),(\ref{eq:3})
is therefore the following: for all $0\le i,j< N_h$ and $ 0\le k < N_T$
\begin{equation}
\label{eq:21}
\left\{
  \begin{array}[c]{llr}
\frac {u^{k+1}_{i,j}- u^{k}_{i,j}} {\dt}  -\nu (\Delta_h u^{k+1})_{i,j}
 + g(x_{i,j}, \left[\nabla_h u^{k+1}\right]_{i,j} )&= F(m^{k}_{i,j}),\\
\frac {m^{k+1}_{i,j}- m^{k}_{i,j}} {\dt} +\nu (\Delta_h m^{k})_{i,j}
+\cT_{i,j}(u^{k+1},m^k)&=0,
\end{array}\right.
\end{equation}
with the initial and terminal conditions
\begin{equation}
  \label{eq:22}
u_{i,j}^0=u_0(x_{i,j}),\quad \quad m_{i,j}^{N_T}=  \frac 1 {h^2} \int_{|x-x_{i,j}|_{\infty}\le h/2} m_T(x) dx,\quad\quad   0\le i,j< N_h.\end{equation}
The following theorem was proved in \cite{MR2679575} (using essentially Brouwer's fixed point theorem and estimates on the solutions of the discrete Bellman equation):
\begin{theorem}
  \label{sec:exist-discr-probl-1}
Assume   that ($\mathbf{g_1}$)--($\mathbf{g_3}$) hold, that $u_0$ is  continuous on $\T^2$ and that $m_T \in L^1 (\T^2)$ is a probability density, i.e. $m_T\ge 0$ and $\int_{\T ^ 2}   m_T(x) dx=1$;  then (\ref{eq:21})--(\ref{eq:22}) has a solution such that $m^n\in \cK_h$,  $\forall n$.
\end{theorem}

Since (\ref{eq:21})-(\ref{eq:22}) has exactly the same structure as the continuous problem (\ref{eq:1})-(\ref{eq:3}), uniqueness has been obtained in \cite{MR2679575} 
with the same arguments as in \cite{MR2271747}:
\begin{theorem}\label{sec:summary-1}
  Assume  that  ($\mathbf{g_1}$)--($\mathbf{g_4}$) hold and that $F$ is nondecreasing 
then  (\ref{eq:21})--(\ref{eq:22}) has a unique solution.
\end{theorem}
\begin{remark}
  \label{sec:summary-2}
Efficient algorithms for solving system (\ref{eq:21})-(\ref{eq:22}) require special efforts, essentially because of the forward-backward structure already discussed above.
We refer to \cite{MR2679575}  for the description of possible algorithms and numerical results.
\end{remark}

\section{Running assumptions and statement of the main result}
\label{sec:statement}

We now summarize the assumptions that will be made in the whole work.
\begin{itemize}
\item $u_0$ is a continuous function on $\T^2$
\item $m_T$  is a  nonnegative function in $L^\infty (\T^2)$ such that $\int_{\T^2} m_T(x)dx=1$
\item $F$ is a continuous function on $\R^+$, which is bounded from below.
\item The Hamiltonian\footnote{il carattere $C^1$ rispetto a  $x$ dove viene realmente usato ? Idem per la $g$} $(x,p)\mapsto H(x,p)$ is assumed to be convex with respect to $p$ and  $\cC^1$ regular w.r.t. $x$ and $p$.
\item The discrete Hamiltonian $g$ satisfies  ($\mathbf{g_1}$)-($\mathbf{g_4}$) and  the further assumption
\begin{description}
\item ($\mathbf{g_5}$) There exist positive constants $c_1, c_2,c_3,c_4$ such that
  \begin{eqnarray}
    \label{eq:27}
g_q(x,q)\cdot q -g(x,q)&\ge& c_1 |g_q(x,q)|^2 - c_2,\\
\label{eq:31}
|g_q(x,q)|&\le &c_3 |q| + c_4.
  \end{eqnarray}
\end{description}
\end{itemize}

\noindent Take for example $g$ as in (\ref{eq:10})~(\ref{eq:11}). It is clear that $g_q(x,q)\cdot q= \beta G(q_1^-,q_2^+,q_3^-,q_4^+ )$,
hence $g_q(x,q)\cdot q -g(x,q)= (\beta -1 )  G(q_1^-,q_2^+,q_3^-,q_4^+ ) - \cH(x)$. Since 
$|g_q(x,q)|^2= \beta^2  \left( G(q_1^-,q_2^+,q_3^-,q_4^+ )\right) ^ {2 \frac {\beta-1} \beta}$, 
 we see that ($\mathbf{g_5}$) holds if $1\leq \beta\le 2$.

\vskip0.5em
We can now state the main result of this article, which establishes the convergence of the  solutions of the finite difference scheme towards a weak solution of the continuous mean field games system.

\begin{theorem}\label{main} Let $(u^n), (m^n)$ be  a solution of the discrete system (\ref{eq:21})-(\ref{eq:22}) and $u_{h,\dt}$, $m_{h,\dt}$ be the piecewise constant functions which take  the values $u_{i,j}^{n+1} $ and $m_{i,j}^{n}$, respectively, in $(t_n,t_{n+1})\times (ih - h/2, ih+h/2) \times (jh -h/2, jh+h/2)$.  There exists a subsequence of $h$ and $\dt$ (not relabeled) and   functions $\tilde u$, $\tilde m$, which belong to $L^\alpha(0,T;W^{1,\alpha}(\T^2))$ for any $\alpha\in [1,\frac43)$, such that $u_{h,\dt} \to \tilde u$ and $m_{h,\dt} \to \tilde m$ in   $L^\beta(Q)$ for all $\beta\in [1,2)$,  and $(\tilde u, \tilde m)$ is a weak solution to the system (\ref{eq:1})-(\ref{eq:3}) in the following sense:
\begin{itemize}
\item[(i)] $H(\cdot, D\tilde u) \in L^1(Q)$,  $\tilde mF(\tilde m)\in L^1(Q)$,  $\tilde m[H_p(\cdot ,D\tilde u)\cdot D\tilde u-H(\cdot,D\tilde u)]\in L^1(Q)$
\item[(ii)] $(\tilde u, \tilde m)$ satisfies (\ref{eq:1})-(\ref{eq:2})  in the sense of distributions
\item[(iii)]  $\tilde u, \tilde m \in C^0([0,T]; L^1(\T^2))$ and $\tilde u|_{t=0}= u_0$, $\tilde m|_{t=T}=m_T$ .
\end{itemize}


\vskip0.4em
\begin{remark} 
We recall that, if $F$ is nondecreasing and $p\mapsto H(x,p)$ is strictly convex at infinity, it is proved in \cite{MR3305653} that weak solutions are unique whenever $H$ satisfies the structure conditions
$$
 \begin{array}{rl}
  H_p(t,x,p)\cdot p  \geq r\, H(t,x,p)  - \, \gamma &
\\
|H_p(t,x,p)| \leq \beta\,  (1+ |p|^{r-1})  &  
\\
H(t,x,p) \geq \alpha |p|^r- \gamma  &
  \end{array}
$$
for some $r\in (1,2]$ and some positive constant $\alpha,\beta,\gamma$. 

Therefore, in this case   the convergence established in  the above theorem holds for the whole sequence, and not only for a subsequence. 
\end{remark}

\end{theorem}

\section{A priori estimates}
\label{sec:priori-estimates}

\subsection{Norms and semi-norms}
\label{sec:some-notation-first}
It is useful to define the following norms and semi-norms:
\\
for a grid function $v\equiv (v_{i,j})_{i,j}$, we define 
  \begin{eqnarray}
\label{eq:17}
\|v\|_{L^s(\T_h^2)}&=&  \left(h^2 \sum_{i,j} |v_{i,j}|^s\right)^{\frac 1 s}, \\
\label{eq:23}
|v|_{W^{1,s}(\T_h^2)}&=&  \left(h^2 \sum_{i,j}  \left((D^+_1 v_{i,j})^2+ (D^+_2 v_{i,j})^{2}\right)^{\frac s  2} \right)^{\frac 1 s}
\\
\|v\|_{W^{1,s}(\T_h^2)}&=&\left(\|v\|_{L^s(\T_h^2)}^s+ |v|_{W^{1,s}(\T_h^2)}^s \right)^{\frac 1  s},
  \end{eqnarray}
where $D^+_1 v$ and $D^+_2 v$ are defined in (\ref{eq:8}).
We shall also write $|v|_{H^{1}(\T_h^2)}=|v|_{W^{1,2}(\T_h^2)}$, $\|v\|_{H^{1}(\T_h^2)}=\|v\|_{W^{1,2}(\T_h^2)}$, and define the  discrete $L^2$ scalar product:
\begin{displaymath}
  (v,w)_{L^2(\T_h^2)}=  h^2 \sum_{i,j} v_{i,j}w_{i,j}.
\end{displaymath}

We recall the discrete Sobolev inequality: for any $s<\infty$, there exists a constant $C$ such that for any grid function $v$,
\begin{displaymath}
  \|v\|_{L^s(\T_h^2)}\le C\left ( \|v\|_{L^2(\T_h^2)} + |v|_{H^1(\T_h^2)}\right)  .
\end{displaymath}
For $s> 1$, we define the  dual norm  $\|v\|_{W^{-1,s'}(\T_h^2)}$,  $\frac 1 s +\frac 1 {s'}=1$ by
 \begin{displaymath}
   \|v\|_{W^{-1,s'}(\T_h^2)}= \sup_{w\not=0} \frac { (v,w)_{L^2(\T_h^2)} }{\|w\|_{W^{1,s}(\T_h^2)}}.  
 \end{displaymath}
Define $Q_{h,\dt}= \dt\{ 0,\dots, N_T-1\}\times \T_h^2$. 
For a function $w$ defined on $Q_{h,\dt}$, $w\equiv (w^n_{i,j})_{i,j}$, $0\le n\le N_T$, we define for $s\in [1,+\infty)$,
\begin{eqnarray}
  \label{eq:24}
\|w\|_{L^s(Q_{h,\dt})}&= &\left(\dt \sum_{n=0}^{N_T} \|w^n\|_{L^s(\T_h^2)}^s\right)^{\frac 1 s}.
\end{eqnarray}


\subsection{First estimates}
\label{sec:first-estimates}
Hereafter, the constants appearing in the a priori estimates, for example $c$, $C$,  are independent of $h$ and $\dt$.
In this paragraph, we state the first a priori estimates stemming from the structure of the system.
Although we have already given the set of running assumptions,
 we think that it may be useful to specify which assumptions are really required by each  particular result.

\begin{lemma}\label{sec:priori-estim-discr-4}
  Under Assumptions ($\mathbf{g_1}$) and ($\mathbf{g_3}$), if $F$ is  
bounded from below by a constant $\underline F$, $u_0$ is continuous on $\T ^2$,
then for all $i,j,n$,
\begin{equation}
  \label{eq:40}
u_{i,j}^n \ge  \underline u - T \left( \underline F - \max_{x\in \T^2} H(x,0)\right)^- ,
\end{equation}
 where $ \underline u= \min_{x\in \T^2} u_0(x)$.
\end{lemma}

\begin{lemma}\label{sec:priori-estim-discr-6}
Under Assumptions ($\mathbf{g_1}$),  ($\mathbf{g_3}$) and  ($\mathbf{g_5}$),  if
$F$ is   bounded from below by  $\underline F$,
$u_0$ is continuous on $\T ^2$ and $m_T$ is bounded from above by $\bar m_T$,
then there exists a constant $C$ such that 
\begin{eqnarray}
   \label{eq:28}
h^2 \dt \sum_{k=0}^{ N_T-1} \sum_{i,j}  m^k_{i,j}   \left | g_q(x_{i,j}, \left[\nabla_h u^{k+1}\right]_{i,j}) \right|^2 \le C,
\\\label{eq:41}
h^2 \dt \sum_{k=0}^{ N_T-1} \sum_{i,j}  g(x_{i,j}, \left[\nabla_h u^{k+1}\right]_{i,j}) \le C,
 \\
\label{eq:42}
 h^2 \dt \sum_{k=0}^{ N_T-1} \sum_{i,j}  m^k_{i,j}  F (m^k_{i,j}) \le C.
\end{eqnarray}
\end{lemma}
\begin{proof}
Consider  $\tilde u_{i,j}^n= n \dt F(\bar m_T)$ 
for all $i,j,n$.
We get immediately
\begin{equation}
   \label{eq:37}
\frac {\tilde u^{n+1}_{i,j}- \tilde u^{n}_{i,j}} {\dt} 
 -\nu (\Delta_h \tilde u^{n+1})_{i,j}
=F(\bar m_T) 
\end{equation}
Subtract (\ref{eq:37}) from (\ref{eq:18}) and multiply the resulting equation by
 $ m_{i,j}^n-\bar m_T$.  
Similarly, multiply (\ref{eq:20})  
 by $ u^{n+1}_{i,j}-\tilde u^{n+1}_{i,j}$. Adding the two 
resulting identities and summing with respect to $n$, one gets:
\begin{equation}
   \label{eq:39}
   \begin{split}
 & \ds h^2 \dt \sum_{k=0}^{ N_T-1} \sum_{i,j}  m^k_{i,j}    \left (
g_q(x_{i,j}, \left[\nabla_h u^{k+1}\right]_{i,j})\cdot \left[\nabla_h u^{k+1}\right]_{i,j}
- g(x_{i,j}, \left[\nabla_h u^{k+1}\right]_{i,j})  \right ) \\
& \ds +h^2 \dt \sum_{k=0}^{ N_T-1} \sum_{i,j}  \bar m_T   g(x_{i,j}, \left[\nabla_h u^{k+1}\right]_{i,j}) \\
& \ds +h^2 \dt \sum_{k=0}^{ N_T-1} \sum_{i,j} 
 (m^k_{i,j} -  \bar m_T )      ( F (m^k_{i,j})- F( \bar m_T )) \\
= & (m^{N_T}-\bar m_T , u^{N_T} - T\, F(\bar m_T))_{L^2(\T_h^2)}
 - (m^{0} -  \bar m_T, u^{0} )_{L^2(\T_h^2)}.
\end{split}
\end{equation}
\begin{enumerate}
\item Since $m^{N_T}-\bar m_T$ is nonpositive with a bounded mass, 
 and since $u^n$ is bounded from below, see (\ref{eq:40}), the term 
$ (m^{N_T}-\bar m_T , u^{N_T} - T\, F(\bar m_T))_{L^2(\T_h^2)}$
 in the right hand side of (\ref{eq:39}) is bounded from above by a  constant independent of $h$ and $\dt$.
\item It is straightforward to see that
 $(m^{0} -  \bar m_T, u^{0} )_{L^2(\T_h^2)}\le (1+ \bar m_T) \|u_0\|_{\infty}$.
\item Since $F$ is continuous, there exists a constant $c$ such that 
$F(t)\le \frac 1{2\bar m_T} t F(t)+c$, $\forall t\ge 0$. Hence,
\begin{displaymath}
h^2 \dt \sum_{k=0}^{ N_T-1} \sum_{i,j}   (m^k_{i,j} -  \bar m_T )       F (m^k_{i,j})\ge
 \frac 1 2 h^2 \dt \sum_{k=0}^{ N_T-1} \sum_{i,j}  m^k_{i,j} F (m^k_{i,j}) -c.
\end{displaymath}
\item Finally, $h^2 \dt \sum_{k=0}^{ N_T-1} \sum_{i,j} 
 (m^k_{i,j} -  \bar m_T )     F( \bar m_T ) = T (1-\bar m_T)F( \bar m_T ) $
\end{enumerate}
From these observations, (\ref{eq:28}), (\ref{eq:41}) and (\ref{eq:42}) follow from (\ref{eq:39}) and  (\ref{eq:27}).
\end{proof}

\subsection{A priori estimates from the discrete Fokker-Planck equation}
\label{sec:priori-estim-discr}
The following estimates for the  Fokker-Planck equation may have their own interest:
\begin{lemma}
Assume ($\mathbf{g_1}$) and $(\mathbf{g_3})$.
  Let $\psi$ be a non decreasing and concave function defined on $\R_+$. 
For any grid function $v= (v_{i,j})$, any positive grid function $m=(m_{i,j})$ and any positive number $\eta$,
\begin{equation}
  \label{eq:4}
  \begin{array}[c]{l}
    \ds \sum_{i,j} \cT_{i,j}(v,m) \psi(m_{i,j}) 
\le  \ds \frac \eta {2} \sum_{i,j} [\nabla_h \psi(m)]_{i,j}\cdot [\nabla_h m]_{i,j}  \\
 \ds + \frac 1 {2\eta} \sum_{i,j}  m^2_{i,j}   \psi'(m_{i,j}) \left( \frac {\partial g} {\partial q_1} (x_{i,j}, [\nabla_h v]_{i,j}) \right)^2  1_{\{ m_{i+1,j}> m_{i,j}\} } \\
 \ds + \frac 1 {2\eta} \sum_{i,j}  m^2_{i,j}   \psi'(m_{i,j}) \left( \frac {\partial g} {\partial q_2} (x_{i,j}, [\nabla_h v]_{i,j}) \right)^2  1_{\{ m_{i-1,j}> m_{i,j}\} }\\
 \ds + \frac 1 {2\eta} \sum_{i,j}  m^2_{i,j}   \psi'(m_{i,j}) \left( \frac {\partial g} {\partial q_3} (x_{i,j}, [\nabla_h v]_{i,j}) \right)^2  1_{\{ m_{i,j+1}> m_{i,j}\} } \\
 \ds + \frac 1 {2\eta} \sum_{i,j}  m^2_{i,j}   \psi'(m_{i,j}) \left( \frac {\partial g} {\partial q_4} (x_{i,j}, [\nabla_h v]_{i,j}) \right)^2  1_{\{ m_{i,j-1}> m_{i,j} \}}.
  \end{array}
\end{equation}
In particular, if $m$ does not vanish, then for $\psi(z)=\ln(z)$, 
\begin{equation}
\label{eq:5}
  \begin{array}[c]{ll}
    &\ds \sum_{i,j} \cT_{i,j}(v,m) \ln(m_{i,j}) 
\\
\le & \ds \frac \eta {2} \sum_{i,j} \left[\nabla_h \ln(m)\right]_{i,j}\cdot \left[\nabla_h m\right]_{i,j} +
  \frac 1 {2\eta} \sum_{i,j}   m_{i,j} \left | g_q(x_{i,j}, [\nabla_h v]_{i,j}) \right|^2.
  \end{array}
\end{equation}
\end{lemma}
\begin{proof}
 By the definition of $\cT$, we can split the sum $S= \sum_{i,j} \cT_{i,j}(v,m) \psi(m_{i,j})$ as follows:
 \begin{displaymath}
   S= - \sum_{i,j}    m_{i,j} \nabla_q g (x_{i,j},[\nabla_h v]_{i,j}) \cdot [\nabla_h \psi(m)]_{i,j}=S_1+S_2+S_3+S_4,
 \end{displaymath}
where
\begin{displaymath}
  \begin{array}[c]{rcl}
    S_1&=&\ds  -\frac 1 h \sum_{i,j}    m_{i,j} 
\frac {\partial g} {\partial q_1} (x_{i,j}, [\nabla_h v]_{i,j})  (\psi(m_{i+1,j})-\psi(m_{i,j})),
\\
 S_2&=&\ds - \frac 1 h \sum_{i,j}    m_{i,j} 
 \frac {\partial g} {\partial q_2} (x_{i,j}, [\nabla_h v]_{i,j})  (\psi(m_{i,j})-\psi(m_{i-1,j})) ,
\\
  S_3&=&\ds  -\frac 1 h \sum_{i,j}    m_{i,j} 
\frac {\partial g} {\partial q_3} (x_{i,j}, [\nabla_h v]_{i,j})  (\psi(m_{i,j+1})-\psi(m_{i,j})),
\\
 S_4&=&\ds  -\frac 1 h \sum_{i,j}    m_{i,j} 
 \frac {\partial g} {\partial q_4} (x_{i,j}, [\nabla_h v]_{i,j})  (\psi(m_{i,j})-\psi(m_{i,j-1})) .
\end{array}
\end{displaymath}
It is enough to focus on $S_1$  since the same arguments can be used for the other sums.
Since $g$ is nonincreasing w.r.t. $q_1$, 
\begin{displaymath}
    S_1 \le \ds  -\frac 1 h \sum_{i,j}    m_{i,j} 
\frac {\partial g} {\partial q_1} (x_{i,j}, [\nabla_h v]_{i,j})  (\psi(m_{i+1,j})-\psi(m_{i,j}))_+ .
\end{displaymath}
Since $\psi$ is nondecreasing, if $m_{i+1,j}>m_{i,j} $, the factor $  (\psi(m_{i+1,j})-\psi(m_{i,j}))_+$ can be rewritten
\[ \left(\frac{\psi(m_{i+1,j})-\psi(m_{i,j})} {m_{i+1,j}-m_{i,j}}\right)^{\frac 1 2}  \Bigl(
(\psi(m_{i+1,j})-\psi(m_{i,j})) (m_{i+1,j}-m_{i,j})  \Bigr)^{\frac 1 2}. \]
 Since $\psi$ is nondecreasing and concave, $m_{i+1,j}>m_{i,j} $ implies that
\[0\le \frac{\psi(m_{i+1,j})-\psi(m_{i,j})} {m_{i+1,j}-m_{i,j}}\le \psi'(m_{i,j} ) .\]
Hence, if  $m_{i+1,j}\ge m_{i,j} $, then
\[  (\psi(m_{i+1,j})-\psi(m_{i,j}))_+ 
\le   \Bigl( \psi'(m_{i,j} )  
(\psi(m_{i+1,j})-\psi(m_{i,j})) (m_{i+1,j}-m_{i,j})  \Bigr)^{\frac 1 2},\]
which implies that
\begin{displaymath}
  \begin{array}[c]{rcl}
       S_1 &\le &\ds   \frac 1 {2\eta } \sum_{i,j}  m^2_{i,j}   \psi'(m_{i,j})    
\left( \frac {\partial g} {\partial q_1} (x_{i,j}, [\nabla_h v]_{i,j}) \right)^2  1_{\{ m_{i+1,j}> m_{i,j}\} }    \\ & &\ds
+    \frac \eta {2h^2} \sum_{i,j}    (\psi(m_{i+1,j})-\psi(m_{i,j}))  (m_{i+1,j}-m_{i,j}) 1_{\{ m_{i+1,j}> m_{i,j}\} }.
  \end{array}
\end{displaymath}
\end{proof}

\begin{lemma}
  \label{sec:priori-estim-discr-1}
Assume ($\mathbf{g_1}$) and $(\mathbf{g_3})$.
  Let $\psi$ be a non decreasing and concave function defined on $\R_+$. 
For any positive grid functions $(m^k_{i,j})$, $k=0,\dots,N_T$,
\begin{equation}
\label{eq:6}
 \sum_{k=n+1}^{ N_T}  \sum_{i,j}   m^k_{i,j}     (\psi(m^{k}_{i,j}) - \psi(m^{k-1}_{i,j}))
 \ge  -\sum_{k=n+1}^{ N_T}  \sum_{i,j}   m^k_{i,j}  \psi'( m^k_{i,j} ) (m^{k-1}_{i,j} -m^{k}_{i,j})_+.
\end{equation}
If $m^k \in \cK_h$ for all $k\in \{0,\dots, N_T\}$ and does not vanish, then 
\begin{equation}
  \label{eq:7}
  \begin{split}
  h^2 \sum_{k=n}^{N_T-1} \sum_{i,j} \left( m^{k+1}_{i,j} -m^{k}_{i,j}\right) \ln(m^k_{i,j})
  \le  h^2   \sum_{i,j}  m^{N_T}_{i,j}\ln(m^{N_T}_{i,j}) -  h^2   \sum_{i,j}  m^{n}_{i,j}\ln(m^{n}_{i,j}) +1  .
  \end{split}
\end{equation}
\end{lemma}
\begin{proof}
Since $\psi$ is non decreasing, 
  \begin{displaymath}
     \sum_{k=n+1}^{ N_T}  \sum_{i,j}   m^k_{i,j}   ( \psi(m^{k}_{i,j}) - \psi(m^{k-1}_{i,j})) \ge 
 \sum_{k=n+1}^{ N_T}  \sum_{i,j}   m^k_{i,j}   ( \psi(m^{k}_{i,j}) - \psi(m^{k-1}_{i,j})) 1_{\{m^{k}_{i,j}< m^{k-1}_{i,j} \} }.
  \end{displaymath}
From the concavity of $\psi$, if $m^{k}_{i,j}<m^{k-1}_{i,j}$, then $\psi(m^{k}_{i,j}) - \psi(m^{k-1}_{i,j})\ge \psi'(m^{k}_{i,j}) (m^{k}_{i,j} -m^{k-1}_{i,j})$. Then (\ref{eq:6}) follows from the last two points. \\
Let us  turn to (\ref{eq:7}): for any $\epsilon>0$, 
  \begin{displaymath}
     \begin{split}
  h^2  \sum_{k=n}^{N_T-1} \sum_{i,j} \left( m^{k+1}_{i,j} -m^{k}_{i,j}\right) \ln(m^k_{i,j}+\epsilon)
 = &h^2   \sum_{i,j}  m^{N_T}_{i,j}\ln(m^{N_T}_{i,j}+\epsilon) -  h^2   \sum_{i,j}  m^{n}_{i,j}\ln(m^{n}_{i,j}+\epsilon) \\ 
&
 -h^2   \sum_{k=n+1}^{ N_T}  \sum_{i,j}   m^k_{i,j}     (\ln(m^{k}_{i,j}+\epsilon) - \ln(m^{k-1}_{i,j}+\epsilon)),
\end{split}
\end{displaymath}
and (\ref{eq:6}) with $\psi(z)=\ln(z+\epsilon)$ yields  
\begin{displaymath}
  \begin{split}
 -h^2   \sum_{k=n+1}^{ N_T}  \sum_{i,j}   m^k_{i,j}     (\ln(m^{k}_{i,j}+\epsilon) - \ln(m^{k-1}_{i,j}+\epsilon)) &\le  
h^2 \sum_{k=n+1}^{ N_T}  \sum_{i,j}   \frac {m^{k}_{i,j} }{ m^{k}_{i,j}+\epsilon} (m^{k-1}_{i,j} -m^{k}_{i,j})_+ \\ &\le  
h^2 \sum_{k=n+1}^{ N_T}  \sum_{i,j}   (m^{k-1}_{i,j} -m^{k}_{i,j})_+ \le 1,    
  \end{split}
\end{displaymath}
where the last estimate comes from the fact that the grid functions $m^k$ all belong to $\cK_h$.
Hence, 
\begin{displaymath}
  h^2 \sum_{k=n}^{N_T-1} \sum_{i,j} \left( m^{k+1}_{i,j} -m^{k}_{i,j}\right) \ln(m^k_{i,j} +\epsilon)
  \le  h^2   \sum_{i,j}  m^{N_T}_{i,j}\ln(m^{N_T}_{i,j}+\epsilon) -  h^2   \sum_{i,j}  m^{n}_{i,j}\ln(m^{n}_{i,j}+\epsilon) +1  .
\end{displaymath} and
 (\ref{eq:7}) is obtained by letting $\epsilon$ tend to $0$.
\end{proof}
\begin{lemma}
  \label{sec:priori-estim-discr-3}
 If $m^{N_T}\in \cK_h$  and ($\mathbf{g_1}$) ($\mathbf{g_3}$) hold, then
there exists a constant $C$ such that,
 for any number $\eta$, $0<\eta<\nu$, 
a solution $ (m^n_{i,j})$ of (\ref{eq:20})  satisfies
\begin{equation}
  \label{eq:14}
  \begin{array}[c]{ll}
    & \ds
 \max_{n}  h^2   \sum_{i,j}  m^{n}_{i,j} |\ln(m^{n}_{i,j})| + 
(\nu -\eta)   
\dt \sum_{k=0}^{ N_T-1} \left| \sqrt{m^k}\right|^2_{H^1( \T_h^2)}
 \\ \le & \ds C + h^2   \sum_{i,j}  m^{N_T}_{i,j} |\ln(m^{N_T}_{i,j})| 
 + \frac {h^2 \dt} {2\eta}  \sum_{k=0}^{ N_T-1} \sum_{i,j}  m^k_{i,j}   \left | g_q\left(x_{i,j}, \left[\nabla_h u^{k+1}\right]_{i,j}\right) \right|^2.
  \end{array}
\end{equation}
For all $\alpha\in [1,2)$, there exists a constant $c$ such that 
\begin{equation}
  \label{eq:16}
  \begin{split}
 & \|m\|_{L^\alpha(Q_{h,\dt})}^{\alpha}
 \\ \le &c\left(1+ h^2   \sum_{i,j}  m^{N_T}_{i,j} |\ln(m^{N_T}_{i,j})| +h^2 \dt \sum_{k=0}^{ N_T-1} \sum_{i,j}  m^k_{i,j}   
\left | g_q(x_{i,j}, \left[\nabla_h u^{k+1}\right]_{i,j}) \right|^2 \right)  .
  \end{split}
\end{equation}
\end{lemma}
\begin{proof}
\paragraph{Step 1}
Take $\epsilon>0$ and consider $\hat m_{i,j}^n=m_{i,j}^n+\epsilon$. Note that $\hat m_{i,j}^n>0$ for all $i,j,n$.
  Multiply the second equation of (\ref{eq:21}) by $\ln(\hat m_{i,j}^n)$ and sum for all $i,j$ and $k=n,\dots, N_T-1$:
  \begin{displaymath}
    \begin{split}
0= &       h^2 \sum_{k=n}^{N_T-1} \sum_{i,j} \left( \hat m^{k+1}_{i,j} -\hat m^{k}_{i,j}\right) \ln(\hat m^k_{i,j})
-\frac {\nu h^2 \dt} 2 \sum_{k=n}^{N_T-1} \sum_{i,j} \left[\nabla_h \hat m^k\right]_{i,j}\cdot \left[\nabla_h \ln(\hat m^k)\right]_{i,j}
\\ & + h^2 \dt  \sum_{k=n}^{N_T-1} \sum_{i,j} \cT_{i,j}(u^{k+1}, m^k) \ln( \hat m^k_{i,j}) .
    \end{split}
  \end{displaymath}
From (\ref{eq:5}) and (\ref{eq:7}), we deduce that
\begin{displaymath}
  \begin{split}
&    h^2   \sum_{i,j}  \hat m^{n}_{i,j}\ln(\hat m^{n}_{i,j}) +  
(\nu-\eta) 
\frac {h^2 \dt} 2 \sum_{k=n}^{N_T-1} \sum_{i,j} \left[\nabla_h \hat m^k\right]_{i,j}\cdot \left[\nabla_h \ln(\hat m^k)\right]_{i,j} \\
\le  &  1+ h^2   \sum_{i,j}  \hat m^{N_T}_{i,j}\ln(\hat m^{N_T}_{i,j})  
 + \frac { h^2 \dt} {2\eta} \sum_{k=n}^{N_T-1} \sum_{i,j}  \hat m^k_{i,j}   \left | g_q(x_{i,j}, \left[\nabla_h u^{k+1}\right]_{i,j}) \right|^2,
  \end{split}
\end{displaymath}
and since $\hat m_{i,j}^k \ln(\hat m^{k}_{i,j})\ge -e^{-1} $,
\begin{equation}
  \label{eq:15}
  \begin{array}[c]{ll}
    & \ds
 \max_{n}  h^2   \sum_{i,j}  \hat m^{n}_{i,j} |\ln(\hat m^{n}_{i,j})| + \frac {\nu -\eta} 2
 h^2 \dt   \sum_{k=0}^{ N_T-1}  \sum_{i,j} \left[ \nabla_h \hat m^k\right]_{i,j}\cdot \left[\nabla_h \ln(\hat m^k)\right]_{i,j}
 \\ \le & \ds C + h^2   \sum_{i,j}  \hat m^{N_T}_{i,j} |\ln(\hat m^{N_T}_{i,j})| 
 + \frac {h^2 \dt} {2\eta}  \sum_{k=0}^{ N_T-1} \sum_{i,j}  \hat m^k_{i,j}   \left | g_q(x_{i,j}, \left[\nabla_h u^{k+1}\right]_{i,j}) \right|^2.
  \end{array}
\end{equation}
Consider now the quantity $| \sqrt{\hat m^k}|^2_{H^1(\T_h^2)}$, i.e.
\begin{displaymath}
 | \sqrt{\hat m^k}|^2_{H^1(\T_h^2)}=  \sum_{i,j} \left(\sqrt{\hat m^k_{i+1,j}}-\sqrt{\hat m^k_{i,j}}\right)^2 +  
\sum_{i,j} \left(\sqrt{\hat m^k_{i,j+1}}-\sqrt{\hat m^k_{i,j}}\right)^2.
\end{displaymath}
Since $\hat m^k_{i,j}>0$, we can write
  $\left(\sqrt{\hat m^k_{i+1,j}}-\sqrt{\hat m^k_{i,j}}\right)^2 = \hat m^k_{i,j} \left( \sqrt{1+ h \frac {(D_1^+ \hat m^k)_{i,j}} {\hat m^k_{i,j}}}-1 \right)^2$ where 
$(D_1^+ \hat m^k)_{i,j}$ is defined in (\ref{eq:8}). Since the inequality $ (\sqrt{1+z}-1)^2\le z \ln (1+z)$ holds for any number $z\ge -1$, we infer that  \begin{displaymath}
  \begin{split}
  \left(\sqrt{\hat m^k_{i+1,j}}-\sqrt{\hat m^k_{i,j}}\right)^2  &\le  h(D_1^+ \hat m^k)_{i,j}  \ln(1+ h \frac {(D_1^+ \hat m)_{i,j}} {\hat m^k_{i,j}})\\
  &=     h(D_1^+ \hat m^k)_{i,j} \left(  \ln(\hat m^k_{i+1,j} )- \ln(\hat m^k_{i,j})\right) \\ &=   h^2(D_1^+ \hat m^k)_{i,j} (D_1^+ \ln(\hat m^k))_{i,j}.
  \end{split}
\end{displaymath}
Since  the same kind of estimate holds for $ \left(\sqrt{\hat m^k_{i,j+1}}-\sqrt{\hat m^k_{i,j}}\right)^2$, we obtain that
\begin{equation}
  \label{eq:13}
 | \sqrt{\hat m^k}|^2_{H^1(\T_h^2)}\le h^2 \sum_{i,j} \left(D_h \hat m^k\right)_{i,j}\cdot \left(D_h \ln(\hat m^k)\right)_{i,j}  ,
\end{equation}
and the fact that $\hat m$ satisfies (\ref{eq:14}) follows from (\ref{eq:15}) and (\ref{eq:13}).\\
Let us now prove  (\ref{eq:16}): consider $\alpha\in [1,2)$: 
there exists a unique number $p\ge 1$ such that $\frac 1 \alpha = \frac 1 2 + \frac 1 {2p}$: 
we have the interpolation inequality $\|\hat m^k\|_{L^ \alpha (\T_h^2)}\le \|\hat m^k\|^{\frac 1 2}_{L^1 (\T_h^2)} 
 \|\hat m^k\|^{\frac 1 2}_{L^p (\T_h^2)}$. But $ \|\hat m^k\|_{L^1 (\T_h^2)}=1+\epsilon$ and 
$ \|\hat m^k\|^{\frac 1 2}_{L^p (\T_h^2)} =   \|\sqrt{\hat m^k}\|_{L^{2p} (\T_h^2)}$.
From the discrete Sobolev inequalities, we deduce that
\begin{displaymath}
  \|\hat m^k\|^\alpha_{L^ \alpha (\T_h^2)}\le  (1+\epsilon) ^{\frac \alpha 2} \|\sqrt{\hat m^k}\|^\alpha_{L^{2p} (\T_h^2)}\le
 C \|\sqrt{\hat m^k}\|^\alpha_{H^{1} (\T_h^2)} \le C\left (1 + \|\sqrt{\hat m^k}\|^2_{H^{1} (\T_h^2)}   \right),
\end{displaymath}
which yields that $\hat m$ satisfies (\ref{eq:16}) by summing for all $k$ and using (\ref{eq:14}).
 \paragraph{Step 2}
We obtain that $m$ satisfies  (\ref{eq:14}) and (\ref{eq:16}) by letting $\epsilon$ tend to $0$.
\end{proof}
\begin{corollary}
  \label{sec:priori-estim-discr-2}
With the same assumptions as in Lemma~\ref{sec:priori-estim-discr-3}, for any $\alpha\in [1,4/3)$,
there exists a constant $c$ such that 
\begin{equation}
\label{eq:26}
  \begin{split}
 & \|D_h m\|_{L^{\alpha} (Q_{h,\dt})}^{\alpha} +
\dt \sum_{k=0}^{N_T-1} \left \| \frac {m^{k+1}-m ^k} {\dt} \right \|^\alpha _{W^{-1,\alpha}(\T_h^2) )} 
  \\ \le & c\left(1+ h^2   \sum_{i,j}  m^{N_T}_{i,j} |\ln(m^{N_T}_{i,j})| +h^2 \dt \sum_{k=0}^{ N_T-1} \sum_{i,j}  m^k_{i,j}   
\left | g_q(x_{i,j}, \left[\nabla_h u^{k+1}\right]_{i,j}) \right|^2 \right)  .
  \end{split}
\end{equation}
\end{corollary}

\begin{proof}
Take $\alpha\in [1,4/3)$.
 We start by observing that \[\|D_h m\|_{L^{\alpha} (Q_{h,\dt})}^{\alpha} \le
 C h^2 \dt \sum_{k=0}^{N_T -1} \sum_{i,j} |D_1^+ m^k|_{i, j}^{\alpha} +|D_2^+ m^k|_{i, j}^{\alpha}. \]
Let us  estimate $ \sum_{k=0}^{N_T -1} \sum_{i,j} |D_1^+ m^k|_{i, j}^{\alpha}$ by some quantity depending on
\\ $ \sum_{k=0}^{N_T -1} \sum_{i,j} (D_1^+ m^k)_{i,j}(D_1^+ \ln (m^k))_{i,j}$: a H{\"o}lder inequality yields that
 \begin{displaymath}
   \begin{split}
       &  \sum_{k=0}^{N_T -1} \sum_{i,j} |D_1^+ m^k|_{i, j}^{\alpha} \\ \le &\left(  \sum_{k=0}^{N_T -1} \sum_{i,j}
 |D_1^+ m^k|_{i, j}^2 \frac {(D_1^+ \ln(m^k))_{i, j}}{(D_1^+ m^k)_{i, j}} \right)^{\frac \alpha 2}  \left(   
\sum_{k=0}^{N_T -1} \sum_{i,j}
\left (\frac {(D_1^+ m^k)_{i, j}}{(D_1^+ \ln(m^k))_{i, j}}\right)^{\frac \alpha {2-\alpha} }    \right)^{1-\frac \alpha 2}\\
 =&  \left(  \sum_{k=0}^{N_T -1} \sum_{i,j} (D_1^+ m^k)_{i, j}(D_1^+ \ln(m^k))_{i, j} \right)^{\frac \alpha 2} 
 \left(   
\sum_{k=0}^{N_T -1} \sum_{i,j}
\left (\frac{(D_1^+ m^k)_{i, j}} {(D_1^+ \ln(m^k))_{i, j}}\right)^{\frac \alpha {2-\alpha} }    \right)^{1-\frac \alpha 2}.
\end{split}
\end{displaymath}
Standard calculus yields that
 $  \frac {(D_1^+ m^k)_{i, j}} {(D_1^+ \ln(m^k))_{i, j}} \le \max ( m^k_{i,j}, m^k_{i+1,j} )\le  m^k_{i,j}+  m^k_{i+1,j} $, therefore
\begin{displaymath}
 \begin{split}
       & h^2 \dt  \sum_{k=0}^{N_T -1} \sum_{i,j} |D_1^+ m^k|_{i, j}^{\alpha} \\ \le &
 h^2 \dt 
 \left(  \sum_{k=0}^{N_T -1} \sum_{i,j} (D_1^+ m^k)_{i, j}(D_1^+ \ln(m^k))_{i, j} \right)^{\frac \alpha 2} 
 \left(   
\sum_{k=0}^{N_T -1} \sum_{i,j} 
( m^k_{i,j}+ m^k_{i+1,j})^{\frac \alpha {2-\alpha} }     \right)^{1-\frac \alpha 2}\\
\le &  C \|m\|_{L^{\frac \alpha {2-\alpha}     }(Q_{h,\dt})}^{ \frac \alpha 2}
\left(  h^2 \dt  \sum_{k=0}^{N_T -1} \sum_{i,j} (D_1^+ m^k)_{i, j}(D_1^+ \ln(m^k))_{i, j} \right)^{\frac \alpha 2}  .
\end{split}
\end{displaymath}
Note that $1\le \frac \alpha {2-\alpha}  <2$.
The estimate on $ \|D_h m\|_{L^{\alpha} (Q_{h,\dt})}^{\alpha}$ follows by using the same argument for bounding $ \sum_{k=0}^{N_T -1} \sum_{i,j} |D_2^+ m^k|_{i, j}^
\alpha$ and then  (\ref{eq:15}) and (\ref{eq:16}).\\
From (\ref{eq:20}), we deduce that for all grid functions $w$, 
\begin{displaymath}
\begin{split}
  &  h^2 \sum_{i,j}    w_{i,j} \frac {m^{n+1}_{i,j}- m^{n}_{i,j}} {\dt} 
\\ = & \nu  h^2  \sum_{i,j}  (\nabla_h m^{n})_{i,j}\cdot  (\nabla_h w)_{i,j}
  +   h^2  \sum_{i,j}    m^{n}_{i,j}  g_q(x_{i,j}, [\nabla_h u^{n+1}]_{i,j}) \cdot  [\nabla_h w]_{i,j}.    
  \end{split}  
\end{displaymath}
  Hence
\begin{displaymath}
  \begin{split}
&\left (\frac {m^{n+1}- m^{n}} {\dt}, w\right )_{L^2  (\T_h^2) } 
 \le      \nu  |m^n|_{W^{1, \alpha}(\T_h^2)} |w|_{W^{1, \alpha'}(\T_h^2)}\\ & + 
 \left( h^2 \sum_{i,j}    m^{n}_{i,j}  |g_q(x_{i,j}, [\nabla_h u^{n+1}]_{i,j})|^2\right)^{\frac 1 2} 
\|\sqrt {m^n}\|_{ L^{\frac {2\alpha}{2-\alpha}} (\T_h^2 ) }  |w|_{W^{1, \alpha'}  (\T_h^2 )},
  \end{split}
\end{displaymath}
and
\begin{equation}
  \label{eq:32}
  \begin{split}
   &\left( \dt \sum_{k=0}^{N_T-1} \left \| \frac {m^{k+1}-m ^k} {\dt} \right \|^\alpha _{W^{-1,\alpha}(\T_h^2) } \right)^{\frac 1 \alpha}
\le \nu \|D_h m\|_{L^{\alpha} (Q_{h,\dt})}\\ & +
 \left( h^2 \dt  \sum_{k=0}^{ N_T-1} 
\sum_{i,j}  m^k_{i,j}   \left | g_q\left(x_{i,j}, \left[\nabla_h u^{k+1}\right]_{i,j}\right) \right|^2
\right)^{\frac 1 2} \| m\|^{\frac 1 2}_{{L^{\frac {\alpha}{2-\alpha}} (Q_{h,\dt})} }.
  \end{split}
\end{equation}
Note that $\frac {\alpha}{2-\alpha}<2$. From (\ref{eq:16}), 
\begin{equation}
  \label{eq:33}
  \begin{split}
       &\| m\|^{\frac 1 2}_{{L^{\frac {\alpha}{2-\alpha}} (Q_{h,\dt})} } \\ \le &
c\left(1+ h^2   \sum_{i,j}  m^{N_T}_{i,j} |\ln(m^{N_T}_{i,j})| +h^2 \dt \sum_{k=0}^{ N_T-1} \sum_{i,j}  m^k_{i,j}   
\left | g_q(x_{i,j}, \left[\nabla_h u^{k+1}\right]_{i,j}) \right|^2 \right)^{\frac {2-\alpha} {2\alpha}} .
  \end{split}
\end{equation}
The desired estimate on $\dt \sum_{k=0}^{N_T-1} \left \| \frac {m^{k+1}-m ^k} {\dt} \right \|^\alpha _{W^{-1,\alpha}(\T_h^2) }   $ follows from 
(\ref{eq:32})-(\ref{eq:33}) and from the estimate on
$ \|D_h m\|_{L^{\alpha} (Q_{h,\dt})}^{\alpha}$.
\end{proof}

Collecting the above results together with Lemma~\ref{sec:priori-estim-discr-6}, we obtain the following conclusion:
\begin{theorem}
  \label{sec:priori-estim-from-1}
  If $F$ is continuous and  bounded from below by a constant $\underline F$,
if ($\mathbf{g_1}$), ($\mathbf{g_3}$), ($\mathbf{g_5}$) hold, if  $u_0$ is continuous,
then there exists a constant $C$ such that a solution $(u,m)$ of (\ref{eq:21})-~(\ref{eq:22}) 
 satisfies (\ref{eq:28})-(\ref{eq:42}),
 and for all $\alpha\in (1,4/3)$,
  \begin{equation}\label{eq:34}
 \|D_h m\|_{L^{\alpha} (Q_{h,\dt})}^{\alpha} +
\left( \dt \sum_{k=0}^{N_T-1} \left \| \frac {m^{k+1}-m ^k} {\dt} \right \|^\alpha _{W^{-1,\alpha}(\T_h^2) )} \right)^{\frac 1 \alpha}
\le C.
  \end{equation}
\end{theorem}

\section{$L^1$-compactness results}
\label{sec:l1-comp-results-3}
In this section we prove the $L^1$-compactness of $D_hu$ whenever the discrete heat equation has bounded $L^1$ data. More precisely, we assume that $u= (u^n)_{n=0,\dots,N_T}$ satisfies
\begin{equation}\label{eql1}
\ds   \frac {u^{n+1}_{i,j}- u^{n}_{i,j}} {\dt}  -\nu (\Delta_h u^{n+1})_{i,j} =   f_{i,j}^n
\end{equation}
for all $0\le i,j< N_h$  and all $n$, $0\le n < N_T$,
where the data $f=(f_{i,j}^n)$ and  the initial conditions $u^0= (u_{i,j}^0)$  are supposed to satisfy
\begin{equation}\label{l1bound}S
\|u ^0\|_{L^1(\T_h^2)} + \|f \|_{L^1(Q_{h,\dt})} \leq c
\end{equation}
for some $c$ independent of $h$ and $\dt$. 
In what follows, we reconstruct functions on $Q$ from the grid functions $u$, and we prove the convergence of these functions 
as $h$ and $\dt$ tend to $0$, at least for subsequences. Lemma \ref{l1-comp} below is concerned with  piecewise constant functions 
built using $u$. It is similar to results that can be found  in Gallou{\"e}t et al, 
see e.g. \cite{MR1652593,MR1804748,MR2904585} in the context of finite volume methods.
 Lemma~\ref{sec:l1-comp-results} deals with approximations of the gradient with respect to $x$.
It  seems new to the best of our knowledge and may have an independent interest. 
\begin{lemma}\label{l1-comp} 
Let $u_{h,\dt}$ be the piecewise constant function which takes the value $u_{i,j}^{n+1} $ in $(t_n,t_{n+1})\times (ih - h/2, ih+h/2) \times (jh -h/2, jh+h/2)$.  There exists a subsequence of $h$ and $\dt$ (not relabeled) and a function $\tilde u$ such that $u_{h,\dt} \to \tilde u$ in 
 in  $L^\beta(Q)$ for all $\beta\in [1,2)$.
Moreover, $\tilde u\in  L^\alpha(0,T;W^{1,\alpha}(\T^2))$ for any $\alpha\in [1,\frac43)$, and  there exist  a bounded Radon measure $\tilde \mu$ in $Q$ and a bounded Radon measure $\tilde \mu_0$ in $\T^2$ such that $\tilde u$ is the unique solution of  
\begin{equation}
  \label{eq:29}
\begin{cases}
\partial_t \tilde u -\nu \Delta \tilde u= \tilde \mu & \hbox{in $Q$,} \\
\tilde u(0,\cdot)= \tilde 
\mu^0 & \hbox{in $\T^2$}.
\end{cases}  
\end{equation}
\end{lemma}
\begin{proof}
  Using the  $L^1$ bounds on the data, we may show with the same argument as in \S~\ref{sec:priori-estim-discr} that
  \begin{displaymath}
\|u\|_{L^\beta(Q_{h,\dt})}+ \left (\dt \sum_{ n=1} ^ {N_T}   | D_hu^n |^\alpha  _{L^\alpha(\T_h^2)}\right) ^{\frac 1 \alpha} \leq c
  \end{displaymath}
for any $\beta\in [1,2)$ and $\alpha\in [1,\frac43)$. From this estimate and (\ref{eql1}), we deduce that $\dt \sum_{ n=1} ^ {N_T}   \|  \frac {u^{n+1}-u ^n} {\dt} \|  _{W^{-1,\alpha}(\T_h^2) } $
 is uniformly bounded. \\
Recall that $u_{h,\dt}$ is the piecewise constant function which takes the value $u_{i,j}^{n+1} $ in $(t_n,t_{n+1})\times (ih - h/2, ih+h/2) \times (jh -h/2, jh+h/2)$.  We can apply the discrete Aubin-Simon lemma in \cite{MR2904585} (Theorem 3.1): up to the extraction of a subsequence,  $u_{h,\dt}$ converges to a function $\tilde u$ in $L^1(Q)$, and in fact in  $L^\beta(Q)$ for all $\beta\in [1,2)$.
Moreover, $\tilde u\in  L^\alpha(0,T;W^{1,\alpha}(\T^2))$ for any $\alpha\in [1,\frac43)$.
\\
Let  $f_{h,\dt}$  be the piecewise constant  function on  which takes the values $f_{i,j}^{n}$ in $(t_n,t_{n+1})\times (ih - h/2, ih+h/2) \times (jh -h/2, jh+h/2)$. Up to the extraction of a subsequence,    $f_{h\dt}$ converges in the weak-$*$ topology to some bounded Radon measure $\tilde \mu$ on $Q$. Call $u^0_h$ the  piecewise constant function on $\T^2$ which takes the values $u^0_{i,j}$ 
 in $(ih - h/2, ih+h/2) \times (jh -h/2, jh+h/2)$.
We may assume that $u^0_{h}$ converges to  a bounded measure $\tilde \mu^0$ on $\T^2$. 
In particular, testing (\ref{eql1}) with smooth functions and passing to the limit, this implies that $\tilde u$ satisfies
 \begin{displaymath}
\int_0^T \int_{\T^2} \tilde u\left( -\vfi_t-\nu \Delta \vfi\right)\, dxdt= \int_0^T\int_{\T^2} \vfi\, d\tilde\mu + \int_{\T^2} \vfi(0)d\tilde\mu^0\,,   
 \end{displaymath}
for every $\vfi\in C^2(\overline Q)$ such that $\vfi(T)=0$. Notice that $\tilde u$ is  the unique solution of the above weak formulation.
\end{proof}

We now define an approximation of $D\tilde u$ from the grid function $u$. For a real number $z$, let ${ \rm floor}(z)$ be the largest integer that does not exceed $z$,
  ${\rm ceil}(z)$ be the smallest integer that is not less than $z$. 
  Let $\widetilde {D u}_{h,\dt}$ be the piecewise constant function from $Q$ to $\R^2$ which takes the value
\begin{displaymath}
    \left(  \left(D_1^+ u^{n+1}\right)_{{\rm floor}(\frac i 2),  {\rm ceil}(\frac j 2)  },
      \left(D_2^+ u^{n+1}\right)_{{\rm ceil}(\frac i 2 ),  {\rm floor}(\frac j 2)  } \right)
    \end{displaymath}
in $  \left(t_n,t_{n+1}\right)\times \left(i\frac h 2 , (i+1)\frac h 2\right)\times \left( j\frac h 2 , (j+1)\frac h 2\right)$.
More explicitly, $\widetilde {D u}_{h,\dt}$  takes the value
  \begin{displaymath}
    \begin{array}[c]{lcl}
    \left(  \left(D_1^+ u^{n+1}\right)_{i,j},  \left(D_2^+ u^{n+1}\right)_{i,j}\right)  &\hbox{ in } &\left(t_n,t_{n+1}\right)\times \left(ih, ih+\frac h 2\right)\times \left(jh, jh+\frac h 2\right)
\\
\left(  \left(D_1^+ u^{n+1}\right)_{i,j},  \left(D_2^+ u^{n+1}\right)_{i,j-1}\right)  &\hbox{ in } &\left(t_n,t_{n+1}\right)\times \left(ih, ih+\frac h 2\right)\times \left(jh-\frac h 2, jh\right)
\\
\left(  \left(D_1^+ u^{n+1}\right)_{i-1,j},  \left(D_2^+ u^{n+1}\right)_{i,j-1}\right)  &\hbox{ in }& \left(t_n,t_{n+1}\right)\times \left(ih-\frac h 2, ih\right)\times \left(jh-\frac h 2, jh\right)
\\
\left(  \left(D_1^+ u^{n+1}\right)_{i-1,j},  \left(D_2^+ u^{n+1}\right)_{i,j}\right)  &\hbox{ in }& \left(t_n,t_{n+1}\right)\times \left(ih-\frac h 2, ih\right)\times \left(jh, jh+\frac h 2\right).      
    \end{array}
  \end{displaymath}
\begin{lemma}\label{sec:l1-comp-results}
Up to the extraction of a subsequence, the functions $\widetilde {D u}_{h,\dt}$ converge a.e. to $D\tilde u$ in $Q$, and  in $L^\alpha(Q)$ for any $\alpha\in [1,\frac43)$.
\end{lemma}
\begin{proof}
  Since $\tilde u$ is the unique weak solution of (\ref{eq:29}),
for every sequence of smooth functions $\tilde \mu_\de$ and smooth initial data $\tilde \mu^0_{\de}$ converging to $\tilde \mu$ and to $\tilde \mu^0$  respectively, in the weak-$*$ sense of measures,
(which, for instance, can be constructed by convolution), the smooth solutions $\tilde U_\de$ satisfying
\begin{displaymath}
\begin{cases}
\partial_t \tilde U_\de -\nu \Delta \tilde U_\de= \tilde \mu_\de & \hbox{in $Q$,} \\
\tilde U_\de(0,\cdot)= \tilde 
\mu^0_{\de} & \hbox{in $\T^2$}
\end{cases}  
\end{displaymath}
will converge to $\tilde u$, e.g. in $L^\alpha(0,T; W^{1,\alpha}(\T^2))$ for any $\alpha\in [1,\frac43)$, see e.g. \cite{MR1025884}. 
\\
We now consider the  finite difference approximation 
\begin{equation}\label{eql1de}
\ds   \frac {U^{n+1}_{\de,i,j}- U^{n}_{\de,i,j}} {\dt}  -\nu (\Delta_h U^{n+1}_\de)_{i,j} =   \mu_{\de,i,j}^n,
\end{equation}
with $(U^0_\de)_{i,j}= \tilde \mu^0_{\de}(x_{i,j})$ and $\mu_{\de,i,j}^n=  \tilde \mu_\de(t_n, x_{i,j})$.
\\
Let $\sigma$ be a positive real number: let $T_\sigma$ be the piecewise linear function defined on $\R$ by $T_\sigma(z)= \max(-\sigma,\min(z,\sigma))$. Let the grid function $e$ on $Q_{h,\dt}$ be given by $e_{i,j} ^n = u_{i,j} ^n-U_{\delta,i,j} ^n$.
Define also, for $0\leq i,j<N_h$, 
$$
\cI_j^n=\left\{i\; {\rm{s.t.}} \; 
  \max[ |e_{i+1,j}^n |\,,\, |e_{i,j}^n| ]\le \sigma \right\}
\,;\qquad 
\cJ_i^n=\left\{j\; {\rm{s.t.}} \; 
  \max[ |e_{i,j+1}^n |\,,\, |e_{i,j}^n| ]\le \sigma \right\}.
$$
For any $s: 0<s<1$ we have
$$
h^ 2 \dt \sum_{n=0}^{N_T-1} \sum_{i,j }  |D_h e^{n+1}_{i,j}|^s \leq  h^ 2 \dt \sum_{n=0}^{N_T-1} \sum_{i,j }  |D_1^+ e^{n+1}_{i,j}|^s+ |D_2^+ e^{n+1}_{i,j}|^s\,.
$$
The first term can be estimated as follows
\begin{displaymath}
  \begin{split}
  & h^ 2 \dt \sum_{n=0}^{N_T-1} \sum_{i,j }  |D_1^+ e^{n+1}_{i,j}|^s
  \leq h^ 2 \dt \sum_{n=0}^{N_T-1}\sum_j\sum_{i\in \cI_j^{n+1} }  (D_1^+ e^{n+1}_{i,j}D_1^+ T_\sigma(e^{n+1})_{i,j})^{\frac s2}
  \\
   &  \qquad + h^ 2 \dt \sum_{n=0}^{N_T-1}\sum_j\sum_{i\not\in \cI_j^{n+1} }  |D_1^+ e^{n+1}_{i,j}|^s 
   \\  \leq  & T^{1-\frac s2} \left(h^ 2 \dt \sum_{n=0}^{N_T-1}\sum_{i,j} D_1^+ e^{n+1}_{i,j}D_1^+ T_\sigma(e^{n+1})_{i,j}\right)^{\frac s2}
   \\
   & \qquad + 
\left(h^ 2 \dt \sum_{n=0}^{N_T-1}\sum_{i,j} |D_1^+ e^{n+1}_{i,j}|\right)^{s}   \left( h^ 2 \dt \sum_{n=0}^{N_T-1}\sum_j\sum_{i\not\in \cI_j^{n+1} } 1 \right)^{1-s}
\\
\leq & T^{1-\frac s2} \left(h^ 2 \dt \sum_{n=0}^{N_T-1}\sum_{i,j} D_1^+ e^{n+1}_{i,j}D_1^+ T_\sigma(e^{n+1})_{i,j}\right)^{\frac s2}
\\
& \qquad + 
\left(h^ 2 \dt \sum_{n=0}^{N_T-1}\sum_{i,j} |D_1^+ e^{n+1}_{i,j}|\right)^{s}   \left( h^ 2 \dt \sum_{n=0}^{N_T-1}\sum_{i,j} \frac{|e_{i+1,j}^{n+1} |+|e_{i,j}^{n+1}| }\sigma  \right)^{1-s}.
\end{split}
\end{displaymath}
Similarly we estimate the term with $D_2^+$ using the set $\cJ^n_i$, and overall we deduce that
\begin{equation}\label{presigma}
  \begin{split}
 &  h^ 2 \dt \sum_{n=0}^{N_T-1} \sum_{i,j }  |D_h e^{n+1}_{i,j}|^s
 \leq c\, T^{1-\frac s2} \left(h^ 2 \dt \sum_{n=0}^{N_T-1}\sum_{i,j} D_h e^{n+1}_{i,j}\cdot D_h T_\sigma(e^{n+1})_{i,j}\right)^{\frac s2} 
 \\
& \qquad + c \| D_h e\|_{L^1 (Q_{h,\dt})}^{s} \| e\|_{L^1 (Q_{h,\dt})}^{1-s} \sigma^{-(1-s)} 
\end{split}
\end{equation}
for some constant $c$ only depending on $s$. We estimate the first term from the discrete equation
\begin{displaymath}
\begin{split}
 & \nu\, h^ 2 \dt \sum_{n=0}^{N_T-1}\sum_{i,j} D_h e^{n+1}_{i,j}\cdot D_h T_\sigma(e^{n+1})_{i,j} = h^ 2 \dt \sum_{n=0}^{N_T-1}\sum_{i,j} (f^n_{i,j}-\mu_{\de,i,j}^n) T_\sigma(e^{n+1}_{i,j}) 
 \\
 & \qquad - h^ 2 \dt \sum_{n=0}^{N_T-1}\sum_{i,j} \frac{e^{n+1}_{i,j}- e^{n}_{i,j}}{\dt} T_\sigma(e^{n+1}_{i,j})
 \end{split}
\end{displaymath}
which implies, using that $(x-y)T_\sigma(x)\geq \Theta_\sigma(x)-\Theta_\sigma(y)$ for the nonnegative and convex function $\Theta(s)= \int_0^s T_\sigma(r)dr$, 
$$
 \nu\, h^ 2 \dt \sum_{n=0}^{N_T-1}\sum_{i,j} D_h e^{n+1}_{i,j}\cdot D_h T_\sigma(e^{n+1})_{i,j} \leq \sigma \left( \|f-\mu_{\de}\|_{L^1(Q_{h,\dt})} + \|e^0\|_{L^1(\T^2_h)}\right)\,.
 $$
Therefore, we deduce from (\ref{presigma})
\begin{displaymath}
   \begin{split}
 &  h^ 2 \dt \sum_{n=0}^{N_T-1} \sum_{i,j }  |D_h e^{n+1}_{i,j}|^s
 \leq c\, T^{1-\frac s2} \sigma^{\frac s2} \left( \|f-\mu_{\de}\|_{L^1(Q_{h,\dt})} + \|e^0\|_{L^1(\T ^2_ h)}\right)^{\frac s2} 
 \\
& \qquad + c \| D_h e\|_{L^1 (Q_{h,\dt})}^{s} \| e\|_{L^1 (Q_{h,\dt})}^{1-s} \sigma^{-(1-s)} \,.
\end{split}
\end{displaymath}
Taking the minimum  of the right hand side w.r.t. $\sigma$, and using the $L^1$ bounds for $\mu_\de$, $\mu_{0\de}$ and the data in (\ref{l1bound}),  we see that 
$$
\| |D_h e|^s\|_{L^1 (Q_{h,\dt})} \leq c  \| D_h e\|_{L^1 (Q_{h,\dt})}^{\theta s} \| e\|_{L^1 (Q_{h,\dt})}^{\theta(1-s)}
$$
for some $c$ and $\theta$ depending on $s$ but not on $h$ or $\de$. Recalling the definition of $e$, and the estimate on the discrete gradient, we have proved that
$$
\| |D_h u- D_h U_{\de} |^s\|_{L^1 (Q_{h,\dt})} \leq c  \| u- U_{\de}\|_{L^1 (Q_{h,\dt})}^{\theta(1-s)}\,.
$$
Hence,
\begin{equation}
  \label{eq:38}
\|\, |D_h u- D_h U_{\de} |^s\|_{L^1 (Q_{h,\dt})} \leq c  \| u_{h,\dt}- U_{\de,h,\dt}\|_{L^1 (Q_{h,\dt})}^{\theta(1-s)},
\end{equation}
 where $u_{h,\dt}$ has been defined in Lemma~\ref{l1-comp} and $ U_{\de,h,\dt}$ is the piecewise constant function that takes the value
$U_{\delta,i,j}^{n+1} $ in $(t_n,t_{n+1})\times (ih - h/2, ih+h/2) \times (jh -h/2, jh+h/2)$. 
\\
Let us also define $\widetilde {D U}_{\delta,h,\dt}$ from the grid function $U_\delta$ in a similar way as $\widetilde {D u}_{h,\dt}$: it takes the values
\begin{displaymath}
    \left(  \left(D_1^+ U_\delta^{n+1}\right)_{{\rm floor}(\frac i 2),  {\rm ceil}(\frac j 2)  },
      \left(D_2^+  U_\delta^{n+1}\right)_{{\rm ceil}(\frac i 2 ),  {\rm floor}(\frac j 2)  } \right)
    \end{displaymath}
in $  \left(t_n,t_{n+1}\right)\times \left(i\frac h 2 , (i+1)\frac h 2\right)\times \left( j\frac h 2 , (j+1)\frac h 2\right)$.
Therefore, we see that
\begin{displaymath}
  \begin{split}
& \| \,|\widetilde {Du}_{h,\dt}- D\tilde u  |^s\|_{L^1 (Q)}   \\
 \leq &   \| \,|\widetilde {Du}_{h,\dt}    -  
\widetilde {D U}_{\delta,h,\dt}
|^s\|_{L^1(Q)}
+ \| \,|
\widetilde {D U}_{\delta,h,\dt}
- D\tilde U_\de|^s\|_{L^1(Q)}+ \|\, |D\tilde U_\de-D\tilde u  |^s\|_{L^1 (Q)}
\\
 \leq &  c \| u_{h,\dt}- U_{\de,h,\dt}\|_{L^1 (Q_{h,\dt})}^{\theta(1-s)}
 + \|\,|\widetilde {D U}_{\de,h,\dt}- D\tilde U_\de|^s\|_{L^1(Q)}+ \|\,|D\tilde U_\de-D\tilde u  |^s\|_{L^1 (Q)}\\
\leq & c  \left(\| u_{h,\dt}-\tilde u\|_{L^1 (Q_{h,\dt})}^{\theta(1-s)}   + \|\tilde u - \tilde U_\delta\| _{L^1 (Q_{h,\dt})}^{\theta(1-s)}+   \| \tilde U_\delta    - U_{\de,h,\dt}\|_{L^1 (Q_{h,\dt})}^{\theta(1-s)}\right)
\\  & + \|\,|\widetilde {D U}_{\de,h,\dt}- D\tilde U_\de|^s\|_{L^1(Q)}+ \|\,|D\tilde U_\de-D\tilde u  |^s\|_{L^1 (Q)}
  \end{split}
\end{displaymath}
where we have used (\ref{eq:38}) to obtain the third line.
At fixed $\de$, since $\tilde U_\de$ is a smooth solution of the heat equation, the discrete approximation  $U_{\de, h, \dt} $ converges
to $\tilde U_\de$ in $L^2(Q) $   and $\widetilde {D U}_{\de,h,\dt}$ converges to $D \tilde U_\delta$ in $L^2(Q;\R^2) $.
Using also Lemma \ref{l1-comp}, we get that   
$$
\limsup\limits_{h, \dt \to 0} \| |Du_{h,\dt}- D\tilde u  |^s\|_{L^1 (Q)}   \leq c \| \tilde u- \tilde U_\de\|_{L^1 (Q)}^{\theta(1-s)} + \||D\tilde U_\de-D\tilde u  |^s\|_{L^1 (Q)}\,.
$$
We conclude using the strong convergence of $\tilde U_\de$ to $\tilde u$ in $L^\alpha(0,T; W^{1,\alpha}(\T^2))$ for any $\alpha\in [1,\frac43)$ (see e.g. \cite{MR1025884}, as $\de\to 0$. So
$$
\| \,|\widetilde{Du}_{h,\dt}- D\tilde u  |^s\|_{L^1 (Q)}  \mathop{\to}\limits^{\dt,h\to 0} 0
$$
which in particular implies that $\widetilde{Du}_{h,\dt}$ converges to $ D\tilde u$ a.e. in $Q$ and  then, by Vitali's theorem, in $L^\alpha(Q)$ for any $\alpha\in [1,\frac43)$.
\end{proof}

\begin{remark}
  \label{sec:l1-comp-results-2}
As a consequence of Lemma~\ref{sec:l1-comp-results}, for any    $\xi \in \R^2$,  there exists a subsequence of $h$ and $\dt$ (not relabeled) such that the maps $(t,x)\mapsto \widetilde{Du}_{h,\dt} (t, x+h\xi)$
also converge to $ D\tilde u$ a.e. and in $L^\alpha(Q)$ for any $\alpha\in [1,\frac43)$.
\end{remark}

\begin{remark}
  \label{sec:l1-comp-results-1}
Alternative strategies can be used to construct  a function defined on $Q$ from  the grid function $u$. For example, we can  
define $w_{h,\dt}$ as the continuous and piecewise trilinear function on $\bar Q$ which takes the values $u_{i,j}^{\max (1,n)} $
 at $(t_n,x_{i,j})$   and which is trilinear in the rectangles  of the time-space grid $Q_{h,\dt}$.
The advantage of taking $w_{h,\dt}$ instead of $u_{h,\dt}$ is that the former has weakly integrable partial derivatives with respect to the spatial 
variable. Therefore, we can use directly $   Dw_{h,\dt}$ instead of having to define an independent approximation of $D\tilde u$ such as  $\widetilde {Du}_{h,\dt}$.
 It is then possible to prove the following
lemma, which may replace both Lemmas \ref{l1-comp} and \ref{sec:l1-comp-results}:
\begin{lemma}
  There exists a subsequence of $h$ and $\dt$ (not relabeled) and a function $\tilde u$ such that $w_{h,\dt} \to \tilde u$ in 
$ L^\alpha(0,T;W^{1,\alpha}(\T^2))$ for any $\alpha\in [1,\frac43)$ and in $L^\beta(Q)$ for all $\beta\in[1,2)$.
In particular, $w_{h,\dt} \to \tilde u$ and $Dw_{h,\dt} \to D\tilde u$ in $L^1(Q)$ and  almost everywhere in $Q$. 
\end{lemma}
\begin{proof}
  The strategy of proof is  similar except that we may directly use the continuous version of 
the compactness lemma of  Aubin-Simon, see \cite{MR916688}, for the function $w_{h,\dt}$.
\end{proof}
\end{remark}

\section{From the discrete to the continuous system}
\label{sec:dis-to-con}

\subsection{A priori estimates and compactness}
\label{sec:priori-estim-comp}
Let $u_{h,\dt}$ and  $\widetilde {Du}_{h,\dt}$ be the piecewise constant functions defined in  Lemmas \ref{l1-comp} and \ref{sec:l1-comp-results} respectively: up to the extraction of a subsequence, we can assume that  $u_{h,\dt} \to \tilde u$  in  $L^\beta(Q)$ for all $\beta\in [1,2)$
 and that $\widetilde {Du}_{h,\dt}$  a.e. to $D\tilde u$ in $Q$ and  in $L^\alpha(Q)$ for any $\alpha\in [1,\frac43)$.
 \\
Let $m_{h,\dt}$ be the piecewise constant function 
 which takes the value $m_{i,j}^{n} $ in $(t_n,t_{n+1})\times (ih - h/2, ih+h/2) \times (jh -h/2, jh+h/2)$, 
and $\widetilde {D m}_{h,\dt}$  be
the piecewise constant function from $Q$ to $\R^2$ which takes the value
\begin{displaymath}
    \left(  \left(D_1^+ m^{n}\right)_{{\rm floor}(\frac i 2),  {\rm ceil}(\frac j 2)  },
      \left(D_2^+ m^{n}\right)_{{\rm ceil}(\frac i 2 ),  {\rm floor}(\frac j 2)  } \right)
    \end{displaymath}
in $  \left(t_n,t_{n+1}\right)\times \left(i\frac h 2 , (i+1)\frac h 2\right)\times \left( j\frac h 2 , (j+1)\frac h 2\right)$.
From Theorem \ref{sec:priori-estim-from-1}, we may also assume that  $m_{h,\dt} \to  \tilde m$ in $L^1(Q)$  and almost everywhere in $Q$,
and that $\widetilde {D m} _{h,\dt}\to D\tilde m$ weakly in $L^\alpha(Q)$ for any $\alpha\in [1,\frac 43)$.
Moreover, for all $\eta>0$, there exists a constant $c_\eta$ 
such that for all $z\ge 0$, $ F(z)\le \frac {z F(z)}\eta +c_\eta$. This fact and estimate (\ref{eq:42}) yield the equi-integrability of $F(m_{h,\dt})$. By Vitali's theorem,
$F(m_{h,\dt})\to F(\tilde m)$   in $L^1(Q)$.
\\
From the observations above, the piecewise constant function which takes the value 
\begin{displaymath}
   \frac {u^{n+1}_{i,j}- u^{n}_{i,j}} {\dt}  -\nu (\Delta_h u^{n+1})_{i,j} -   F(m^{n}_{i,j}) 
\end{displaymath}
in $(t_n,t_{n+1})\times (ih - h/2, ih+h/2) \times (jh -h/2, jh+h/2)$ converges to $
\frac{\partial \tilde u}{\partial t}-\nu \Delta \tilde u -F(\tilde m)$ in the sense of distributions.

\subsection{Stability of the discrete Bellman equation}
We now pass to the limit in the discrete Bellman equation.

The main difficulty is to handle the nonlinear term $g(x_{i,j}, \left[\nabla_h u^{n+1}\right]_{i,j} )$; 
here  we wish to use the a.e. convergence of the gradients obtained in \S~\ref{sec:l1-comp-results-3}.
 We adapt the method used for continuous problems in \cite{MR766873}.
Note that $\left[\nabla_h u^{n+1}\right]_{i,j}$ is the value taken by the piecewise constant function with values in $\R^4$
\begin{displaymath}
\left(
  \begin{array}[c]{ll}
     \vec{e_1} \cdot \widetilde {D u}_{h,\dt} (\cdot+\frac h 2 \vec{e_1}  )  , 
 \vec{e_1} \cdot \widetilde {D u}_{h,\dt} (\cdot-\frac h 2 \vec{e_1} ), 
 \vec{e_2} \cdot \widetilde {D u}_{h,\dt} (\cdot+\frac h 2 \vec{e_2})  , 
 \vec{e_2} \cdot \widetilde {D u}_{h,\dt} (\cdot- \frac h 2 \vec{e_2} )
  \end{array}\right)
\end{displaymath}
at $(t,x)$ such that  $|x_1-ih|< h/2$, $|x_2-jh|< h/2$, $t_n\le t< t_{n+1}$.
From  the continuity of $g$, 
the consistency assumption and Remark \ref{sec:l1-comp-results-2},
$$
g_{h,\dt} \to g(x, D_1 \tilde u, D_1\tilde u, D_2 \tilde u, D_2 \tilde u)= H(x,D\tilde u) \qquad \mbox{a.e. in $Q$,}
$$
where $g_{h,\dt}$ is the piecewise constant function which take the value 
$g(x_{i,j}, \left[\nabla_h u^{n+1}\right]_{i,j} )$  for $(t,x)$ such that  $|x_1-ih|< h/2$, $|x_2-jh|< h/2$, $t_n\le t< t_{n+1}$.

Let now $\vfi$ be a smooth function on $\T^2$ such that $\vfi\geq 0$, with $\vfi(T)=0$. We multiply the discrete Bellman equation by $\vfi(t_{n+1}, x_{i,j})$ and sum for all $i,j$ and $n=0,\dots, N_T-1$. Since, by convexity, 
$$
g(x,q)\geq g(x,0) + g_q(x,0)\cdot q
$$
the regularity of $g$ w.r.t. $x$ and the $L^1$-compactness of $\widetilde {Du}_{h,\dt}$ allow us to apply Fatou's lemma obtaining 
$$
\liminf\limits_{h\to 0} \, h^ 2 \dt \sum_{n=0}^{N_T-1} \sum_{i,j }g(x_{i,j}, \left[\nabla_h u^{n+1}\right]_{i,j} )\vfi(t_{n+1}, x_{i,j}) \geq  \int_Q H(x,D\tilde u) \vfi\, dxdt\,.
$$
Passing to the limit in the other terms of the equation, we deduce that
$$
-\int_Q \tilde u\,\vfi_t\, dxdt +\nu  \int_Q D\tilde uD\vfi\, dxdt+ \int_Q H(x,D\tilde u) \vfi\, dxdt \leq  \int_{\T^2} u_0\, \vfi(0)\, dx+ \int_Q F(\tilde m)\vfi\, dxdt\,.
$$
We now wish to obtain the reverse inequality, which is the difficult part. We start by noticing that, since the monotonicity assumption implies
$$
g(x,q_1,q_2,q_3,q_4) \leq g(x, -q_1^-, q_2^+, -q_3^-, q_4^+)
$$
from (\ref{eq:27}) and (\ref{eq:31}) and the fact that $g(x,0)$ is bounded, we know there exists $\la>0$ such that
\begin{equation}\label{ggrow}
g(x,q_1,q_2,q_3,q_4) \leq \nu\, \la \left[ 1+ (q_1^-)^2+ (q_2^+)^2 + (q_3^-)^2+ (q_4^+)^2  \right]\,.
\end{equation}
We multiply the discrete Bellman equation by $e^{-\la  u_{i,j}^{n+1}}\vfi(t_{n+1}, x_{i,j})$ and sum for all $i,j$ and $n=0,\dots, N_T-1$. We obtain
\begin{equation}\label{whole}
\begin{split}
& h^ 2 \dt \sum_{n=0}^{N_T-1}\sum_{i,j} \frac{u_{i,j}^{n+1}- u^n_{i,j}}{\dt} \, e^{-\la  u_{i,j}^{n+1}}\vfi(t_{n+1}, x_{i,j})
\\
 & \qquad + \nu\, h^ 2 \dt \sum_{n=0}^{N_T-1}\sum_{i,j} D_h u^{n+1}_{i,j}\cdot D_h (e^{-\la  u_{i,j}^{n+1}}\vfi(t_{n+1}, x_{i,j}))_{i,j} 
 \\
 & \qquad + h^ 2 \dt \sum_{n=0}^{N_T-1}\sum_{i,j}g(x_{i,j}, \left[\nabla_h u^{n+1}\right]_{i,j} ) e^{-\la  u_{i,j}^{n+1}}\vfi(t_{n+1}, x_{i,j})
 \\
 & = h^ 2 \dt \sum_{n=0}^{N_T-1}\sum_{i,j} F(m_{i,j}^n)e^{-\la  u_{i,j}^{n+1}}\vfi(t_{n+1}, x_{i,j}).
  \end{split}
\end{equation}
Since $u$ is uniformly bounded below, the last term converges by dominated convergence, so 
\begin{equation}\label{term4}
\lim\limits_{h\to 0}h^ 2 \dt \sum_{n=0}^{N_T-1}\sum_{i,j} F(m_{i,j}^n)e^{-\la  u_{i,j}^{n+1}}\vfi(t_{n+1}, x_{i,j})  = \int_Q F(\tilde m)e^{-\la \tilde u}\, \vfi\, dxdt
.\end{equation}
By convexity of $s\mapsto e^{-\la s}$ and since $\phi(T,\cdot)=0$, we have
\begin{displaymath}
\begin{split}
& h^ 2 \dt \sum_{n=0}^{N_T-1}\sum_{i,j} \frac{u_{i,j}^{n+1}- u^n_{i,j}}{\dt} \, e^{-\la  u_{i,j}^{n+1}}\vfi(t_{n+1}, x_{i,j}) \\
&\qquad  \leq \frac1\la \, h^ 2 \dt \sum_{n=0}^{N_T-1}\sum_{i,j} \frac{e^{-\la u_{i,j}^{n}}-e^{-\la u_{i,j}^{n+1}}}{\dt} \,  \vfi(t_{n+1}, x_{i,j}) 
\\
 &\qquad = \frac1\la \, h^ 2 \dt \sum_{n=0}^{N_T-1}\sum_{i,j} e^{-\la u_{i,j}^n}\, \frac{\vfi(t_{n+1},x_{i,j})-\vfi(t_n,x_{i,j})}{\dt} + \frac1\la \, h^ 2 \sum_{i,j} e^{-\la u_{i,j}^0}\vfi(0,x_{i,j})\,,
\end{split}
\end{displaymath}
and so, again by dominated convergence,
\begin{equation}\label{term1} 
\begin{split} & 
\limsup\limits_{h\to 0} \quad h^ 2 \dt \sum_{n=0}^{N_T-1}\sum_{i,j} \frac{u_{i,j}^{n+1}- u^n_{i,j}}{\dt} \, e^{-\la  u_{i,j}^{n+1}}\vfi(t_{n+1}, x_{i,j}) 
\\
& \qquad \leq \frac1\la \int_Q e^{-\la u} \vfi_t\, dxdt + \frac1\la \int_{\T^2} e^{-\la u_0}\, \vfi(0)\, dx \,.
\end{split}
\end{equation}
Let us deal now jointly with the second and third term in (\ref{whole}). First we split the energy term according to the sign of $D_1^+(u)_{i,j}$ (and $D_2^+(u)_{i,j}$, respectively); indeed, we can write
\begin{align*}
& D_1^+(u^{n+1})_{i,j} D_1^+(e^{-\la u^{n+1}}\vfi^{n+1})_{i,j} 
\\
& = 
(D_1^+(u^{n+1})_{i,j})^+ \left( e^{-\la u_{i+1,j}^{n+1}}- e^{-\la u_{i,j}^{n+1}}\right) \vfi_{i+1,j}^{n+1} - (D_1^+(u^{n+1})_{i,j})^-  \left( e^{-\la u_{i+1,j}^{n+1}}- e^{-\la u_{i,j}^{n+1}}\right)\vfi_{i,j}^{n+1}
\\
& 
+ (D_1^+(u^{n+1})_{i,j})^+ (\vfi_{i+1,j}^{n+1}-\vfi_{i,j}^{n+1}) e^{-\la u_{i,j}^{n+1}} - (D_1^+(u^{n+1})_{i,j})^- (\vfi_{i+1,j}^{n+1}-\vfi_{i,j}^{n+1}) e^{-\la u_{i+1,j}^{n+1}} 
\end{align*}
and the same for the term with $D_2^+$. Reordering the indexes in the sum, this means that 
the $D_1$ part in the second order term can be read as
\begin{align*}
& \nu\, h^ 2 \dt \sum_{n=0}^{N_T-1}\sum_{i,j} (D_1^+(u^{n+1})_{i-1,j})^+ \frac1h\left( e^{-\la u_{i,j}^{n+1}}- e^{-\la u_{i-1,j}^{n+1}}\right) \vfi_{i,j}^{n+1} 
\\
&
- \nu\, h^ 2 \dt \sum_{n=0}^{N_T-1}\sum_{i,j}  (D_1^+(u^{n+1})_{i,j})^-  \frac1h\left( e^{-\la u_{i+1,j}^{n+1}}- e^{-\la u_{i,j}^{n+1}}\right)\vfi_{i,j}^{n+1}
\\
&\quad +  \nu\, h^ 2 \dt \sum_{n=0}^{N_T-1}\sum_{i,j}  (D_1^+(u^{n+1})_{i,j})^+ D^+_1(\vfi^{n+1})_{i,j} e^{-\la u_{i,j}^{n+1}}
\\
& \qquad - \nu\, h^ 2 \dt \sum_{n=0}^{N_T-1}\sum_{i,j}(D_1^+(u^{n+1})_{i,j})^- D^+_1(\vfi^{n+1})_{i,j} e^{-\la u_{i+1,j}^{n+1}} 
\,,
\end{align*}
which is equal to 
\begin{align*}
& \nu\, h^ 2 \dt \sum_{n=0}^{N_T-1}\sum_{i,j} |(D_1^+(u^{n+1})_{i-1,j})^+|^2 \frac{ e^{-\la u_{i,j}^{n+1}}- e^{-\la u_{i-1,j}^{n+1}}}{ u_{i,j}^{n+1}-  u_{i-1,j}^{n+1}} \,  \vfi_{i,j}^{n+1} 
\\
&
+ \nu\, h^ 2 \dt \sum_{n=0}^{N_T-1}\sum_{i,j} |D_1^+(u^{n+1})_{i,j})^-|^2  \frac{e^{-\la u_{i+1,j}^{n+1}}- e^{-\la u_{i,j}^{n+1}}}{u_{i+1,j}^{n+1}- u_{i,j}^{n+1}}
\vfi_{i,j}^{n+1}
\\
&\quad +  \nu\, h^ 2 \dt \sum_{n=0}^{N_T-1}\sum_{i,j}  \left\{(D_1^+(u^{n+1})_{i,j})^+ e^{-\la u_{i,j}^{n+1}}-(D_1^+(u^{n+1})_{i,j})^- e^{-\la u_{i+1,j}^{n+1}} \right\}D^+_1(\vfi^{n+1})_{i,j} \,.
\end{align*}
We proceed similarly for the part with $D_2$. Therefore, 
\begin{align*}
& \nu\, h^ 2 \dt \sum_{n=0}^{N_T-1}\sum_{i,j} D_h u^{n+1}_{i,j}\cdot D_h (e^{-\la  u_{i,j}^{n+1}}\vfi(t_{n+1}, x_{i,j}))_{i,j} 
 \\
 & \qquad + h^ 2 \dt \sum_{n=0}^{N_T-1}\sum_{i,j}g(x_{i,j}, \left[\nabla_h u^{n+1}\right]_{i,j} ) e^{-\la  u_{i,j}^{n+1}}\vfi(t_{n+1}, x_{i,j})
 \\
 & = \nu\, h^ 2 \dt \sum_{n=0}^{N_T-1}\sum_{i,j}  \left\{(D_1^+(u^{n+1})_{i,j})^+ e^{-\la u_{i,j}^{n+1}} - (D_1^+(u^{n+1})_{i,j})^- e^{-\la u_{i+1,j}^{n+1}} \right\}D^+_1(\vfi^{n+1})_{i,j}  
 \\
 & +\nu\, h^ 2 \dt \sum_{n=0}^{N_T-1}\sum_{i,j}  \left\{(D_2^+(u^{n+1})_{i,j})^+ e^{-\la u_{i,j}^{n+1}} - (D_2^+(u^{n+1})_{i,j})^- e^{-\la u_{i,j+1}^{n+1}} \right\}D^+_2(\vfi^{n+1})_{i,j} 
 \\
& + \nu\, h^ 2 \dt \sum_{n=0}^{N_T-1}\sum_{i,j}  |(D_1^+(u^{n+1})_{i-1,j})^+|^2 \frac{ e^{-\la u_{i,j}^{n+1}}- e^{-\la u_{i-1,j}^{n+1}}}{ u_{i,j}^{n+1}-  u_{i-1,j}^{n+1}} \vfi_{i,j}^{n+1} 
\\
&
+ \nu\, h^ 2 \dt \sum_{n=0}^{N_T-1}\sum_{i,j}  |D_1^+(u^{n+1})_{i,j})^-|^2  \frac{e^{-\la u_{i+1,j}^{n+1}}- e^{-\la u_{i,j}^{n+1}}}{u_{i+1,j}^{n+1}- u_{i,j}^{n+1}}
\vfi_{i,j}^{n+1}
\\
& + \nu\, h^ 2 \dt \sum_{n=0}^{N_T-1}\sum_{i,j} |(D_2^+(u^{n+1})_{i,j-1})^+|^2 \frac{ e^{-\la u_{i,j}^{n+1}}- e^{-\la u_{i,j-1}^{n+1}}}{u_{i,j}^{n+1}- u_{i,j-1}^{n+1}} \vfi_{i,j}^{n+1} 
\\
&
+ \nu\, h^ 2 \dt \sum_{n=0}^{N_T-1}\sum_{i,j}  |(D_2^+(u^{n+1})_{i,j})^- |^2\, \frac{ e^{-\la u_{i,j+1}^{n+1}}- e^{-\la u_{i,j}^{n+1}}}{u_{i,j+1}^{n+1}-u_{i,j}^{n+1}}\vfi_{i,j}^{n+1}
\\ &
\qquad + h^ 2 \dt \sum_{n=0}^{N_T-1}\sum_{i,j}g(x_{i,j}, \left[\nabla_h u^{n+1}\right]_{i,j} ) e^{-\la  u_{i,j}^{n+1}}\vfi_{i,j}^{n+1}.
 \end{align*}
The first two terms in the right-hand side converge to $\nu \int_0^T \int_{\T_2}  e^{-\lambda \tilde u } D\tilde u\cdot d\phi$
by Lebesgue theorem, since $\widetilde {Du}_{h,\dt}$ converges strongly in $L^1$, $\vfi$ is smooth and $e^{-\la u_h}$ is uniformly bounded and a.e. convergent. 
As far as the remaining terms are concerned, 
using that 
$$
\frac{e^{-\la s}- e^{-\la s'}}{s-s'} \leq -\la e^{-\la \max(s,s')},
$$
and due to  (\ref{ggrow}), 
we observe that the last five terms under summation are bounded above, so that we can again apply Fatou's lemma, on account of the a.e. convergence of $u_{h, \dt}$ and $\widetilde {Du}_{h,\dt}$. Therefore, we conclude that
\begin{equation}\label{term2+3} 
\begin{split} & 
\limsup\limits_{h\to 0} \quad \nu\, h^ 2 \dt \sum_{n=0}^{N_T-1}\sum_{i,j} D_h u^{n+1}_{i,j}\cdot D_h (e^{-\la  u_{i,j}^{n+1}}\vfi(t_{n+1}, x_{i,j}))_{i,j} 
 \\
 & \qquad + h^ 2 \dt \sum_{n=0}^{N_T-1}\sum_{i,j}g(x_{i,j}, \left[\nabla_h u^{n+1}\right]_{i,j} ) e^{-\la  u_{i,j}^{n+1}}\vfi(t_{n+1}, x_{i,j})
 \\
 & \leq \nu \int_Q D\tilde u D\vfi\, e^{-\la \tilde u}\, dxdt -\nu \,\la  \int_Q |D\tilde u|^2\, e^{-\la \tilde u}\, \vfi\, dxdt + H(x,D\tilde u)e^{-\la \tilde u} \,\vfi\, dxdt\,.
\end{split}
\end{equation}
Putting together \rife{term4}-\rife{term1}-\rife{term2+3}, we deduce from \rife{whole} that $\tilde u$ satisfies
\begin{equation}\label{supers}
\begin{split}
& \frac1\la \int_Q e^{-\la \tilde u} \vfi_t\, dxdt + \frac1\la \int_{\T^2} e^{-\la u_0}\, \vfi(0)\, dx+ 
\nu \int_Q D\tilde u D\vfi\, e^{-\la \tilde u}\, dxdt 
\\
&\quad  -\nu \,\la  \int_Q |D\tilde u|^2\, e^{-\la \tilde u}\, \vfi\, dxdt + \int_QH(x,D\tilde u)e^{-\la \tilde u} \,\vfi\, dxdt \geq \int_Q F(\tilde m) e^{-\la \tilde u} \,\vfi\, dxdt 
\end{split}
\end{equation}
for every smooth $\vfi\geq 0$. In order to conclude, we need now to get rid of the exponential in the above inequality \rife{supers}. To this purpose, we first observe that 
\begin{equation}\label{expu}
e^{-\la \tilde u}\in L^2(0,T; H^1(\T^2))\cap L^\infty(Q).
\end{equation}
This can be easily proved obtaining an a  priori estimate on $e^{-\la u_{h,\dt}}$. Indeed, whenever $u$ is a grid function which solves  \rife{eql1} for some data satisfying \rife{l1bound}, we have
\begin{align*}
& \left| h^ 2 \dt \sum_{n=0}^{N_T-1}\sum_{i,j} \frac{u_{i,j}^{n+1}- u^n_{i,j}}{\dt} \, \psi(u_{i,j}^{n+1})
+ \nu\, h^ 2 \dt \sum_{n=0}^{N_T-1}\sum_{i,j} D_h u^{n+1}_{i,j}\cdot D_h \psi(u_{i,j}^{n+1}) \right| \leq C\, \|\psi\|_\infty
\end{align*}
for any bounded real function $\psi(r)$. In particular, if $\psi$ is nondecreasing, this implies
$$
h^ 2   \sum_{n=0}^{N_T-1}\sum_{i,j} \Psi(u_{i,j}^{n+1})- \Psi(u^n_{i,j})+ \nu\, h^ 2 \dt \sum_{n=0}^{N_T-1}\sum_{i,j} D_h u^{n+1}_{i,j}\cdot D_h \psi(u_{i,j}^{n+1}) \leq C\, \|\psi\|_\infty
$$
where $\Psi(s)=\int_0^s \psi(r)dr$. Thus, since $|\Psi(s)|\leq \|\psi\|_\infty |s|$, one gets
$$
\nu\, h^ 2 \dt \sum_{n=0}^{N_T-1}\sum_{i,j} D_h u^{n+1}_{i,j}\cdot D_h \psi(u_{i,j}^{n+1}) \leq C\, \|\psi\|_\infty
$$
where $C$ only depends on the $L^1$-norm of the data. This is the desired  a priori estimate; from which, using Fatou's lemma, we deduce
$$
\int_Q |D\tilde u|^2\, \psi'(\tilde u) \, dxdt \leq C \|\psi\|_\infty\,.
$$
On account of the fact that $u$ is bounded below, we can take for example $\psi(r)= 1-e^{-\mu r}$ to deduce  that $e^{-\mu \tilde u}\in L^2(0,T; H^1(\T^2))$ for any $\mu>0$. 

Thanks to \rife{expu},  inequality  \rife{supers} holds true not only for smooth functions $\vfi$ but also for $\vfi\in  H^1(Q)\cap L^\infty$, through a standard density argument. Moreover, there is no loss of generality in assuming  that $u_0\in H^1(\T^2)$, so we  extend  $\tilde u$  for negative $t$ as identically equal to $u_0$. Then, we choose
$$
\vfi(x,t) = \xi(t) \frac1h\int_{t-h}^{t} e^{\la T_k(\tilde u)(x,s)}\, ds 
$$
where $T_k(r)=\min(r,k)$ and $\xi \in C^1_c[0,T)$. 
%
Using the monotone character of $s\mapsto e^{-\la s}$ we have (see Lemma 2.3
in \cite{MR2119989})
\begin{align*}
& \limsup\limits_{h\to 0} \left\{\frac1\la \int_Q e^{-\la u} \vfi_t\, dxdt + \frac1\la \int_{\T^2} e^{-\la u_0}\, \vfi(0)\, dx\right\} 
\\
& \quad \leq - \int_Q \xi_t \int_0^u e^{-\la (r-T_k(r))}\, dr
- \int_{\T^2} \xi(0) \int_0^{u_0} e^{-\la (r-T_k(r))}\, dr\,.
\end{align*}
Moreover, $\frac1h\int_{t-h}^{t} e^{\la T_k(\tilde u)(x,s)}\, ds$ converges to $e^{\la T_k(u)}$ in $L^2(0,T;H^1(\T^2))$ and weak$-*$ 
in $L^\infty(Q)$, so we can pass to the limit as $h\to 0$ in the remaining terms of \rife{supers}. Finally, we obtain
\begin{align*}
& - \int_Q \xi_t \int_0^u e^{-\la (r-T_k(r))}\, dr
- \int_{\T^2} \xi(0) \int_0^{u_0} e^{-\la (r-T_k(r))}\, dr \\
& \quad + 
\nu \int_Q D\tilde u D\xi\, e^{-\la (\tilde u-T_k(\tilde u))}\, dxdt 
 -\nu \,\la  \int_{\{ \tilde u>k\}} |D\tilde u|^2\, e^{-\la (\tilde u-T_k(\tilde u))}\, \xi\, dxdt 
 \\
&\quad + \int_Q H(x,D\tilde u)e^{-\la (\tilde u-T_k(\tilde u))} \,\xi\, dxdt \geq \int_Q F(\tilde m) e^{-\la (\tilde u-T_k(\tilde u))}\xi \, dxdt .
\end{align*}
We conclude by letting $k\to \infty$, thanks to the dominated convergence theorem:
$$
-\int_Q \tilde u\,\xi_t\, dxdt - \int_{\T^2} u_0\, \xi(0)\, dx+\nu  \int_Q D\tilde uD\xi\, dxdt+ \int_Q H(x,D\tilde u) \xi\, dxdt \geq  \int_Q F(\tilde m) \,\xi \, dxdt \,,
$$
for every $\xi\geq 0$. Since the reverse inequality was already obtained previously, in the end, we proved that $u$ solves 
$$
-\int_Q \tilde u\,\xi_t\, dxdt - \int_{\T^2} u_0\, \xi(0)\, dx+\nu  \int_Q D\tilde uD\xi\, dxdt+ \int_Q H(x,D\tilde u) \xi\, dxdt  = \int_Q F(\tilde m) \,\xi \, dxdt
$$
for every $\xi\in C^1_c([0,T)), \xi\geq 0$, and therefore for every $\xi$. This concludes the proof that $\tilde u$ is a weak solution to the limit equation.

\subsection{Stability of the discrete Fokker-Planck  equation}

We now pass to the limit in the discrete Fokker-Planck equation.

By \rife{eq:28},  the $L^1$-compactness of $m_{h,\dt}$ and of $\widetilde {Du}_{h,\dt}$, 
we deduce the strong convergence in $L^1(Q)$ for the piecewise constant function  which takes the value $m_{i,j}^n \nabla_q g (x_{i,j},[\nabla_h u^{n+1}]_{i,j}) $  for $(t,x)$ such that  $|x_1-ih|< h/2$, $|x_2-jh|< h/2$, $t_n\le t< t_{n+1}$. Moreover, by the consistency assumption we have,
\begin{align*}
& h^2\dt \sum_{n=0}^{N_T-1}\sum_{i,j}    m_{i,j}^n \nabla_q g (x_{i,j},[\nabla_h u^{n+1}]_{i,j}) \cdot [\nabla_h \vfi]_{i,j} \to 
\\
& \qquad \to \int_Q \nabla_q g(x, D_1 \tilde u, D_1\tilde u, D_2 \tilde u, D_2 \tilde u)\cdot (D_1 \vfi, D_1\vfi, D_2 \vfi, D_2 \vfi)dxdt
\\
& \qquad\qquad \qquad = \int_Q m\, \frac{\partial H}{\partial p}(x,D\tilde u)\cdot D\vfi \,dxdt
\end{align*}
Therefore, we can  pass to the limit in the weak formulation and deduce that $m$ is a weak solution of the Fokker-Planck equation.

We notice that the regularity $m[H_p(\cdot ,D\tilde u)D\tilde u-H(\cdot,D\tilde u)]\in L^1(Q)$ follows from inequality (\ref{eq:39}), by using Fatou's lemma. Moreover, we also find that $m|H_p(\cdot ,D\tilde u)|^2\in  L^1(Q)$. The regularity $\tilde u, \tilde m\in C^0([0,T];L^1(\T^2))$ follows from properties of weak solutions,  see \cite{MR3305653}.
\\
Finally, this concludes the proof of Theorem \ref{main}.

%
%

\section*{\bf Acknowledgements}
 The first author  was partially funded  by the ANR projects ANR-12-MONU-0013 and ANR-12-BS01-0008-01.

\bibliographystyle{plain}
\bibliography{MFG}

\end{document}